\documentclass[a4paper, reqno, 11pt]{article}
\usepackage{color}
\makeatletter

\@addtoreset{equation}{section}
\makeatother
\usepackage{setspace}
\usepackage{xcolor}
\usepackage{here}
\usepackage{graphicx}
\usepackage{float}
\usepackage{mathtools}
\usepackage{array}
\usepackage{booktabs}
\usepackage{tabularx}
\usepackage[T1]{fontenc}
\usepackage[hidelinks]{hyperref}
\everymath{\displaystyle}
\usepackage{amsmath}
\usepackage{amssymb}
\usepackage{latexsym}
\usepackage{amsthm}

\newtheorem{Thm}{Theorem}[section]

\newtheorem{Lem}[Thm]{Lemma}
\newtheorem{Prop}[Thm]{Proposition}
\newtheorem{Conj}[Thm]{Conjecture}
\newtheorem{Cor}[Thm]{Corollary}

\theoremstyle{definition}

\newtheorem{Def}[Thm]{Definition}
\newtheorem{Rem}[Thm]{Remark}
\newtheorem{Exa}[Thm]{Example}

\def\AGM{\mathrm{AGM}}

\newcommand{\SPD}{\operatorname{SPD}}

\begin{document}

\title{Dual-connection midpoint matrix means and their Gauss composition}
\author{
  Frank Nielsen\\
  Sony Computer Science Laboratories Inc.\\
  E-mail: \texttt{Frank.Nielsen@acm.org}
  \and
  Kazuki Okamura\\
  Department of Mathematics, Faculty of Science\\
  Shizuoka University\\
  E-mail: \texttt{okamura.kazuki@shizuoka.ac.jp}
}
\date{\today}
\maketitle

\begin{abstract}
Nakamura [J. Comput. Appl. Math. 131 (2001)] proved that the arithmetic--harmonic matrix iteration converges quadratically to the geometric matrix mean, the Riemannian midpoint of the affine-invariant metric on positive-definite matrices.
We consider the more general problem of when a Riemannian midpoint is a Gauss compound mean.
A Riemannian metric and an affine connection determine three midpoint maps associated with the connection, its metric dual, and the Levi--Civita connection.
We show that the Gauss composition of the first two equals the Levi--Civita midpoint if and only if the latter is invariant under one step of the iteration.
We give a sufficient condition for this invariance in terms of an isometry which acts as the point reflection about the Levi--Civita midpoint and exchanges the connection with its dual.
Under this invariance, the iterations converge quadratically for sufficiently close initial pairs.

Nakamura's iteration is an example of this criterion.
The criterion also applies to dual pairs of matrix power means whose Gauss composition is the geometric matrix mean.
We give a dually flat Hessian example which shows that duality alone does not suffice.
We also give non-flat examples with a Euclidean metric and a parallel cubic form in every dimension larger than one.
In dimension one, we characterize the invariance completely by using the Matkowski--Sut\^{o} equation.
For Euclidean dual pairs, we conjecture that the invariance implies that the cubic form is constant.
We prove this conjecture in dimension one and for scalar multiples of a constant cubic form.
\end{abstract}

\medskip
\noindent\textbf{2020 Mathematics Subject Classification.}
Primary 53B12; Secondary 26E60, 39B12, 53B05, 53C22.

\section{Introduction}

\subsection{Gauss's compound means}

Gauss's arithmetic--geometric mean (AGM) is the classical example of a compound mean.
Denote the arithmetic and geometric means by 
 $A(p,q) = (p+q)/2$ and $G(p,q) = \sqrt{pq}$ for $p, q > 0$,  respectively.
 The AGM of $p$ and $q$ is defined as
\[\AGM(p,q) = \lim_{n\rightarrow\infty} a_n = \lim_{n\rightarrow\infty} g_n,\]
 where $a_n = A(a_{n-1},g_{n-1})$ and $g_n = G(a_{n-1},g_{n-1})$ with $a_0 = p$ and $g_0 = q$.
Given two means $A$ and $M$, set $a_0=p$ and $m_0=q$ and define
\[a_n=A(a_{n-1},m_{n-1}), \qquad m_n=M(a_{n-1},m_{n-1}) \quad (n\geq 1).\]
If both sequences converge to a common limit, that limit is called their \emph{Gauss composition}~\cite{daroczy-pales-2002,nielsen-2023} and is denoted by $(A\otimes M)(p,q)$.

The Gauss composition need not belong to the same class as its component means.
For example, the arithmetic and geometric means are homogeneous quasi-arithmetic means, but the AGM is not a quasi-arithmetic mean.
Indeed, only power means are homogeneous quasi-arithmetic means \cite[Theorem 84]{hardy-littlewood-polya-1952}, and the AGM is not a power mean.

\subsection{Geometric matrix compound mean and information geometry}

Nakamura~\cite[Theorems~9 and~10]{nakamura-2001} showed that the matrix arithmetic--harmonic iteration on $\SPD(m)$, which is the space of $m\times m$ symmetric positive-definite matrices, converges quadratically to the affine-invariant geometric matrix mean 
$$ P\#Q = P^{\frac{1}{2}}\, \left(P^{-\frac{1}{2}}\, Q\, P^{-\frac{1}{2}} \right)^{\frac{1}{2}}\, P^{\frac{1}{2}}. $$
This pair of arithmetic and harmonic means has an interpretation in information geometry~\cite{AmariNagaoka2000}.
Let a Riemannian metric $g$ and an affine connection $\nabla$ be fixed.
The connection $\nabla$, its $g$-dual $\nabla^{*}$, and the Levi--Civita connection $\nabla^{g}$ determine three geodesic midpoint means $A$, $M$, and $M_g$.
For the Euclidean connection and the affine-invariant metric, the pair $(A,M)$ is the arithmetic--harmonic pair.
The scalar case is the invariance of the geometric mean under the arithmetic--harmonic pair.
This appears in Sut\^{o}'s 1914 solution of the invariance problem for quasi-arithmetic means~\cite{suto-1914-2}.
See Remark~\ref{rem:suto}.
For multivariate normal distributions, Kobayashi computes Fisher--Rao midpoints by lifting the problem to a horizontal geodesic in a positive-definite matrix manifold, applying the arithmetic--harmonic iteration, and projecting back~\cite{kobayashi2023geodesics}.

The matrix power means of Lim and P\'alfia~\cite{lim-palfia-2012} provide a family containing the arithmetic, geometric, and harmonic means.
For two matrices, their interpretation as points on geodesics of affine connections was established by P\'alfia~\cite[Theorem~12.4]{palfia-2013}.
In Section~\ref{subsec:alpha}, we use these connections to study the Gauss composition of the means with opposite parameters.
We also relate this construction to the quasi-arithmetic matrix means and the metrics studied by Hiai and Petz~\cite{hiai-petz-2009} and Kim~\cite{kim-2018}.

\subsection{Contributions}

We consider whether the Gauss composition of this geometrically determined pair is the Riemannian midpoint, that is, whether $A \otimes M = M_g$.
 
\begin{figure}
\centering
\includegraphics[width=0.75\textwidth]{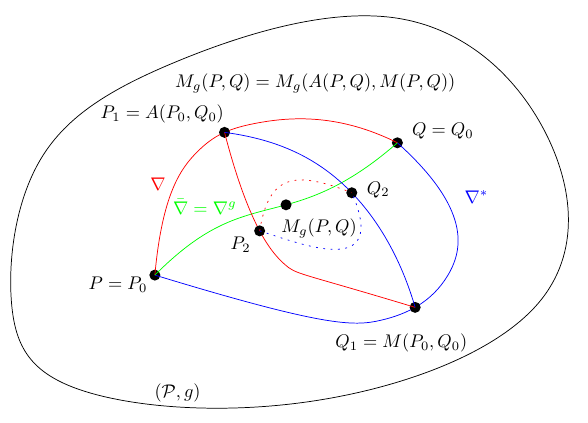}
\caption{Under convergence of the mean-type iteration, midpoint invariance ensures that the Gauss composition equals the Levi--Civita midpoint.\label{fig:match}}
\end{figure}

The pair $(A,M)$ defines the \emph{mean-type mapping}~\cite{matkowski-1999-iterations} 
\[ T(P,Q) \coloneqq \left(A(P,Q), M(P,Q)\right).\]
A single step of the joint iteration is $(P_{n+1},Q_{n+1}) = T(P_{n},Q_{n})$. 
We say that a set $\mathcal{D}$ of pairs is \emph{$T$-invariant} when $T(\mathcal{D}) \subseteq \mathcal{D}$.
On any $T$-invariant domain on which the iteration converges to a common limit, 
Theorem~\ref{thm:invariance} shows that $A \otimes M = M_g$ is equivalent to the midpoint invariance
\begin{equation}\label{eq:Mg-inv}
M_{g} \left(A(P,Q),M(P,Q)\right) = M_{g} (P,Q). 
\end{equation}
See Figure~\ref{fig:match}.

For $C^2$ local midpoint maps, we show that the mean-type iteration converges quadratically for sufficiently close pairs without any further hypothesis (Theorem~\ref{thm:local-existence}).
On a strongly convex normal neighborhood, the invariance implies that the limit is the Levi--Civita midpoint (Theorem~\ref{thm:local-convergence}).
In the affine-invariant example of Section~\ref{sec:nakamura}, the convergence is global with the explicit constant $|\alpha|/2$ (Proposition~\ref{prop:alpha-family}).
We solve the corresponding scalar recursion in closed form and give the number of iterations needed for a prescribed accuracy (Corollary~\ref{cor:alpha-closed-form}). 

We give a sufficient condition for the invariance identity in terms of point symmetry (Theorem~\ref{thm:symmetry}).
For each fixed admissible pair, suppose that there is an isometry on a suitable invariant domain which acts as the point reflection at $M_{g} (P,Q)$ and exchanges $\nabla$ and $\nabla^*$.
If this isometry has no other fixed point in the domain, then the invariance identity \eqref{eq:Mg-inv} holds.
This criterion does not require flatness, affine coordinates, or explicit formulas for the midpoint maps.
It gives a proof by symmetry of the one-step midpoint invariance in Nakamura's iteration.
It also gives the same invariance for the dual pairs in the power-mean family $\nabla^{(\alpha)}$ of $\alpha$-connections on $\SPD(m)$.
The corresponding mean-type iterations converge quadratically to the geometric matrix mean from every initial pair (Proposition~\ref{prop:alpha-family}). 

However, duality alone does not suffice.
We give an explicit dually flat counterexample using a Hessian metric on $\SPD(2)$, for which $A \otimes M  \ne  M_g$ (Section~\ref{sec:counterexample}).
We also consider a Euclidean metric with a parallel cubic form.
This construction gives non-flat examples in every dimension $d \ge 2$ for which the invariance identity holds locally for sufficiently close pairs (Section~\ref{sec:parallel-cubic}).
Thus dual flatness is neither sufficient nor necessary for local midpoint invariance.  

In dimension one, the invariance condition is completely characterized (Section~\ref{sec:one-dim}). 
After an affine change of the primal coordinate and a constant rescaling of the metric, the structure is either trivial or of scalar arithmetic--harmonic--geometric type.
Restriction to suitable common autoparallel curves then yields necessary conditions in higher dimensions.

Finally, we consider Euclidean structures determined by a smooth totally symmetric cubic form $C$.
We conjecture that local midpoint invariance implies that $C$ is constant.
Thus, according to the conjecture, the parallel-cubic examples give all Euclidean dual solutions of the invariance identity (Section~\ref{sec:conjecture}).
By a fourth-order expansion of the geodesic midpoint maps at the diagonal, we obtain a pointwise polynomial necessary condition.
Using this condition, we prove the conjecture in dimension one and for cubic forms that are scalar multiples of a constant one.
We also give an example which shows that the duality between the two connections cannot be dropped.

\subsection{Outline}

The rest of this paper is organized as follows.
In Section~\ref{sec:setting}, we describe the dualistic setting and define the three midpoint means.
In Section~\ref{sec:GC}, we introduce the Gauss composition and the invariance identity, and give the local quadratic-convergence lemma.
Section~\ref{sec:symmetry} is devoted to the point-symmetry criterion.
In Sections~\ref{sec:nakamura}--\ref{sec:parallel-cubic}, we give the three examples.
In Section~\ref{sec:one-dim}, we give the one-dimensional classification.
In Section~\ref{sec:conjecture}, we state the rigidity conjecture for Euclidean dual pairs and prove the fourth-order obstruction.
In Appendix~\ref{app:remarks}, we give further remarks on dualistic structures which are not needed for the main arguments.

\section{The dualistic geometric setting}\label{sec:setting}

Let $(\mathcal{P},g)$ be a Riemannian manifold and let $\nabla$ be an affine connection on $\mathcal{P}$.
The $g$-dual, or conjugate~\cite{nomizu-1992}, connection $\nabla^*$ of $\nabla$ with respect to $g$ is defined by
\begin{equation}\label{eq:duality}
 X[g(Y,Z)]  =  g(\nabla_{X} Y, Z)  +  g(Y, \nabla_X^{*} Z),
\end{equation}
for any vector fields $X$, $Y$, and $Z$. 
Since the metric $g$ is non-degenerate, this identity uniquely determines the dual connection $\nabla^*$.
It also gives $(\nabla^*)^{*} = \nabla$. 
We write $\nabla^{g}$ for the Levi--Civita connection of $g$.

Assume that every relevant pair of points is joined by a unique geodesic for each connection under consideration.
This assumption is a restriction even for the Levi--Civita connection.
The injectivity radius guarantees uniqueness of the minimizing geodesic, but does not exclude longer geodesic segments with the same endpoints.
For the Stiefel manifold with the canonical and Euclidean metrics, the shortest geodesic loops and the injectivity radius are determined in \cite{stoye2026shortest}.
On Hadamard manifolds (Example~\ref{exa:special-cases}(1) below) uniqueness holds globally.
In local constructions, uniqueness is understood among geodesic segments contained in the specified convex normal neighborhood, as in Sections~\ref{sec:GC} and~\ref{sec:parallel-cubic}.
For the Levi--Civita connection on a strongly convex neighborhood, we use the unique minimizing segment.

\begin{Def}[Midpoint mean]\label{def:midpoint-mean}
Let $D$ be an affine connection on $\mathcal{P}$, and let $P$ and $Q$ be joined by a unique affinely parametrized $D$-geodesic $\gamma \colon [0,1] \to \mathcal{P}$ satisfying $\gamma(0) = P$ and $\gamma(1) = Q$, with uniqueness understood in the ambient manifold or in the specified convex normal neighborhood, as above.
The \emph{midpoint mean} of $P$ and $Q$ with respect to $D$ is $M_{D}(P,Q) \coloneqq \gamma(1/2)$.  
We denote the midpoint means of the three connections by
\[ A(P,Q) \coloneqq  M_{\nabla}(P,Q), \; M(P,Q) \coloneqq  M_{\nabla^*}(P,Q), \; M_g(P,Q) \coloneqq  M_{\nabla^g}(P,Q). \]
\end{Def}

The following two properties of $M_{D}$ follow from Definition~\ref{def:midpoint-mean} and the uniqueness assumption.
First, the reversed curve $t \mapsto \gamma(1-t)$ is an affinely parametrized $D$-geodesic which joins $Q$ to $P$ and has midpoint $\gamma(1/2)$.
By uniqueness, it is the geodesic defining $M_{D}(Q,P)$.
Thus $M_{D}$ is \emph{symmetric}, that is, $M_{D}(P,Q) = M_{D}(Q,P)$.
Second, the constant curve at $P$ is a $D$-geodesic joining $P$ to itself. 
By uniqueness it is the geodesic defining $M_{D}(P,P)$, so $M_{D}$ is \emph{idempotent}, that is, $M_{D}(P,P) = P$.
Both properties enter the proof of Theorem~\ref{thm:local-existence}.

In terms of the exponential map of $D$, whenever $Q$ lies in a normal neighborhood of $P$,
\begin{equation}\label{eq:midpoint-exp}
M_{D}(P,Q) = \operatorname{Exp}^{D}_P \left(  \frac{1}{2} \left(\operatorname{Exp}^{D}_P\right)^{-1}(Q)\right),
\end{equation}
where $\operatorname{Exp}^{D}_P$ denotes the exponential map of $D$ at $P$, defined by $\operatorname{Exp}^{D}_P(v) \coloneqq \gamma_{v}(1)$ for the $D$-geodesic $\gamma_{v}$ with $\gamma_{v}(0) = P$ and $\dot{\gamma}_{v}(0) = v$.
Since $\operatorname{Exp}^{D}_P$ is a local diffeomorphism at the origin of $T_{P}\mathcal{P}$, its inverse is well defined on a normal neighborhood of $P$; see \cite[Chapter~III]{kobayashi-nomizu-1963}.

\begin{Rem}[Affine parameter versus arc length]\label{rem:affine-vs-arclength}
In Definition~\ref{def:midpoint-mean}, the midpoint is taken with respect to the affine parameter of $D$, not with respect to $g$-arc length.
For $D = \nabla^{g}$ the affine parameter is proportional to arc length, so $M_g(P,Q)$ is the equidistant point on the Riemannian geodesic joining $P$ and $Q$.
For $\nabla$ and $\nabla^{*}$, however, the $g$-speed of a geodesic need not be constant, and the affine-parameter midpoint generally differs from the arc-length midpoint.
This distinction is needed even in dimension one.
The geodesics of $\nabla$, $\nabla^{*}$, and $\nabla^{g}$ joining two points have the same image, but their parametrizations may differ.
These parametrizations account for the differences between $A$, $M$, and $M_g$ on the scalar rays of Sections~\ref{sec:nakamura} and~\ref{sec:counterexample} and in the one-dimensional classification of Section~\ref{sec:one-dim}.
\end{Rem}

\begin{Exa}\label{exa:special-cases}
(1) On a Hadamard manifold, $M_{g}(P,Q)$ is globally defined and coincides with the unique equal-weight Karcher mean of $P$ and $Q$;
see \cite{lim2015approximations}.\\
(2) If $\nabla = \nabla^E$ is the Euclidean flat connection on a convex matrix domain, then $A(P,Q) = (P + Q)/2$.
\end{Exa}

The mean $A$ is determined by $\nabla$.
The dual connection $\nabla^*$ and its mean $M$ are determined by the pair $(g,\nabla)$, and the Riemannian midpoint $M_g$ is determined by $g$.
Thus, once $g$ and $\nabla$ are fixed, the mean $M$ cannot be chosen independently as the midpoint mean of the $g$-dual connection.
Affine changes of dual coordinates and the usual affine gauge of a Hessian potential do not change the induced midpoint map.

We will use the following two facts about dual pairs.
Their proofs and further remarks on dualistic structures are given in Appendix~\ref{app:remarks}.
First, if both $\nabla$ and $\nabla^{*}$ are torsion-free, then their arithmetic mean $\frac{1}{2}(\nabla + \nabla^{*})$ is torsion-free and $g$-metric, hence equals $\nabla^{g}$, so that
\begin{equation}\label{eq:dual-average}
\nabla^{*} = 2\nabla^{g} - \nabla ;
\end{equation}
moreover, for torsion-free $\nabla$, the dual connection $\nabla^{*}$ is torsion-free if and only if the cubic tensor $\nabla g$ is totally symmetric, which is the setting of statistical manifolds~\cite{lauritzen-1987}.
Second, if $\nabla$ is flat, that is, both curvature-free and torsion-free, then locally there are $\nabla$-affine coordinates $\theta$ in which $\theta\left(A(P,Q)\right) = \frac{1}{2}\left(\theta(P) + \theta(Q)\right)$.
Thus $A$ is a quasi-arithmetic mean in the coordinates $\theta$.
The arithmetic formula of Example~\ref{exa:special-cases}(2) is the case in which the given coordinates are affine.
For a torsion-free connection $\nabla$, nonzero curvature at a point prevents such a coordinate representation on a neighborhood of that point, although $A$ is still a geodesic midpoint mean.
For connections with torsion, midpoint maps depend only on the torsion-free part; see Remark~\ref{rem:curvature-freeness}.
Neither flatness nor a coordinate representation is used in the invariance criterion of Theorem~\ref{thm:symmetry}.

\section{The Gauss composition}\label{sec:GC}

Starting from $P_0  =  P$ and $Q_0  =  Q$, consider the mean-type iteration~\cite{nielsen-2023,nakamura-2001}
\[ P_{n + 1}  =  A(P_n,Q_n), \quad Q_{n + 1}  =  M(P_n,Q_n). \]
If both sequences converge to the same point, we define their Gauss composition by
\[ (A\otimes M)(P,Q) \coloneqq \lim_{n\to\infty} P_n  =  \lim_{n\to\infty} Q_n. \]

\begin{Thm}[Invariance criterion for the Gauss composition to match the Levi--Civita midpoint]\label{thm:invariance}
Let $\mathcal{D} \subset \mathcal P\times\mathcal P$ be a domain on which the midpoint means $A$, $M$, and $M_g$ are all defined, that is, each pair $(P,Q) \in \mathcal{D}$ is joined by a unique geodesic of each of the three connections (Definition~\ref{def:midpoint-mean}). 
Set $T(P,Q) \coloneqq (A(P,Q),M(P,Q))$, and assume that $T(\mathcal{D}) \subseteq \mathcal{D}$. 
For $n=0,1,\ldots$, let $T^{n}$ denote the $n$-fold iterate of $T$, with $T^{0}=\operatorname{id}_{\mathcal D}$, so that $T^{n}(P,Q)=(P_n,Q_n)$. 
Suppose that, for every $(P,Q) \in \mathcal{D}$, there exists a point $L(P,Q) \in \mathcal{P}$ such that
\[ T^n(P,Q) \longrightarrow \left(L(P,Q),L(P,Q)\right), \quad n \to \infty, \]
so that the Gauss composition $(A \otimes M)(P,Q) = L(P,Q)$ is defined for every $(P,Q) \in \mathcal{D}$, and suppose that $M_g$ is also defined at each limit pair $\left(L(P,Q),L(P,Q)\right)$ and is continuous there.
Then $A \otimes M  =  M_g$ on $\mathcal{D}$ if and only if the invariance identity \eqref{eq:Mg-inv} holds on $\mathcal{D}$, that is,  
\[ M_{g}\left(A(P,Q),M(P,Q)\right)  =  M_{g}(P,Q), \quad (P,Q) \in \mathcal{D}. \]
\end{Thm}

\begin{proof}
Fix $(P,Q) \in \mathcal{D}$ and write $L = L(P,Q)$.
Since $T(P,Q) \in \mathcal{D}$ and $T^{n}(T(P,Q)) = T^{n+1}(P,Q) \to (L,L)$, we have $L(T(P,Q)) = L(P,Q)$.
Indeed, removing the first iterate does not change the limit.
Thus $A \otimes M$ is invariant under one application of the mean-type mapping $T$.
If $A \otimes M = M_g$ on $\mathcal{D}$, this is the identity \eqref{eq:Mg-inv} at $(P,Q)$.

Conversely, if \eqref{eq:Mg-inv} holds on $\mathcal{D}$, then $M_g(P_n,Q_n) = M_g(P,Q)$ for all $n$, and the continuity of $M_g$ at $(L,L)$ together with the idempotence $M_g(L,L) = L$ gives
\[ M_g(P,Q)  =  M_g(P_n,Q_n) \longrightarrow M_g(L,L)  =  L, \quad n \to \infty. \]
Hence $(A \otimes M)(P,Q) = L = M_g(P,Q)$.
\end{proof}

Duality of connections alone does not imply the invariance identity \eqref{eq:Mg-inv}. 
Consequently, even though $A$, $M$, and $M_g$ arise from the same dualistic structure, it is possible that $A \otimes M \ne M_g$.

Theorem~\ref{thm:invariance} assumes convergence of the mean-type iteration.
This assumption holds near the diagonal.
For any two sufficiently regular symmetric idempotent maps, the iteration converges quadratically for sufficiently close pairs.
In particular, this applies to the local midpoint maps of any two affine connections.
The existence result does not require a metric, duality, or the invariance identity \eqref{eq:Mg-inv}.
These are used only to identify the limit in Theorem~\ref{thm:local-convergence} below.

Throughout this section, a map $N$ from an open subset of $\mathcal P\times\mathcal P$ to $\mathcal P$ is called \emph{symmetric} if $N(X,Y)=N(Y,X)$ and \emph{idempotent} if $N(X,X)=X$ whenever both sides are defined.
By Definition~\ref{def:midpoint-mean}, local midpoint maps are symmetric and idempotent.
Given such maps $A$ and $M$, we write $T(X,Y)\coloneqq\left(A(X,Y),M(X,Y)\right)$ and $(P_{n+1},Q_{n+1})\coloneqq T(P_n,Q_n)$, as in Theorem~\ref{thm:invariance}.

\begin{Thm}[Local existence and quadratic convergence of the Gauss composition]\label{thm:local-existence}
Let $o\in\mathcal P$, and let $A$ and $M$ be symmetric idempotent maps of class $C^{2}$ from an open neighborhood $\mathcal W$ of $(o,o)$ in $\mathcal P\times\mathcal P$ to $\mathcal P$.
Fix a coordinate chart around $o$ and write $|\cdot|$ for the coordinate norm.
Then there exist a compact set $\mathcal K$ with $\mathcal K\times\mathcal K\subset\mathcal W$, an open neighborhood $\mathcal U\subset\mathcal K$ of $o$, an open set $\mathcal D$ invariant under $(X,Y)\mapsto(Y,X)$ with $T(\mathcal D)\subset\mathcal D\subset\mathcal K\times\mathcal K$, and constants $\varepsilon>0$ and $C>0$ with the following properties, where
\[ \Omega_{\varepsilon}\coloneqq\left\{(X,Y)\in\mathcal U\times\mathcal U\colon |X-Y|<\varepsilon\right\}\subset\mathcal D. \]
For every $(P_0,Q_0)\in\mathcal D$, with $\delta_n\coloneqq|P_n-Q_n|$:\\
(1) $(P_n,Q_n)\in\mathcal D$ and $\delta_{n+1}\le C\delta_n^{2}\le\frac{1}{2}\delta_n$ for every $n\ge0$;\\
(2) there exists $L(P_0,Q_0)\in\mathcal K$ such that $|P_n-L(P_0,Q_0)|\le2\delta_n$ and $|Q_n-L(P_0,Q_0)|\le3\delta_n$ for every $n\ge0$; in particular, $(A\otimes M)(P_0,Q_0)=L(P_0,Q_0)$ is defined.\\
Moreover, the map $L=A\otimes M$ is continuous, symmetric, and idempotent on $\mathcal D$, and\\
(3) $\left|L(X,Y)-\dfrac{A(X,Y)+M(X,Y)}{2}\right|\le\dfrac{2}{3}C^{3}|X-Y|^{4}$ for every $(X,Y)\in\mathcal D$.
\end{Thm}

\begin{proof}
For $N\in\{A,M\}$, let $D_1 N_{(X,Y)}$ and $D_2 N_{(X,Y)}$ denote the partial differentials of $N$ at $(X,Y)$ with respect to the first and the second argument, respectively, and let $D_2^2 N$ denote the second partial differential with respect to the second argument; along the diagonal, $D_1 N_{(X,X)}$ and $D_2 N_{(X,X)}$ are endomorphisms of $T_{X}\mathcal{P}$ because $N(X,X)=X$.
Differentiating the identity $N(X,X)=X$ along the diagonal yields
\[ D_{1} N_{(X,X)}+D_{2} N_{(X,X)}= \operatorname{id}_{T_{X} \mathcal{P}},\]
whereas symmetry gives $D_1 N_{(X,X)}=D_2 N_{(X,X)}$.
Therefore,
\[ D_1 A_{(X,X)} = D_2 A_{(X,X)} = D_1 M_{(X,X)} = D_2 M_{(X,X)} =\frac{1}{2} \operatorname{id}_{T_{X} \mathcal{P}}.\]

Identify points near $o$ with their coordinates, and for $r>0$ let $\mathcal B_r$ denote the open coordinate ball of radius $r$ centered at $o$.
Since $\mathcal W$ is an open neighborhood of $(o,o)$ and $A(o,o)=M(o,o)=o$, there is $r_0>0$ such that $\overline{\mathcal B_{r_0}}$ lies in the chart and $A$ and $M$ are of class $C^{2}$ and take values in the chart on a neighborhood of $\overline{\mathcal B_{r_0}}\times\overline{\mathcal B_{r_0}}$.
Put $\mathcal K\coloneqq\overline{\mathcal B_{r_0}}$,
\[ C_N \coloneqq \frac{1}{2} \sup_{\mathcal K\times\mathcal K} \left\| D_2^2 N \right\| < \infty \quad (N\in\{A,M\}), \]
\[ C\coloneqq\max\{1,C_A+C_M\}, \qquad \kappa\coloneqq\max\{C_A,C_M\}, \]
and choose $\varepsilon>0$ with
\[ C\varepsilon\le\frac{1}{2}, \qquad \kappa\varepsilon\le\frac{1}{2}, \qquad \varepsilon\le\frac{r_0}{8}. \]
Let $(X,Y)\in\mathcal K\times\mathcal K$.
By the convexity of $\mathcal K$, the segment
\[ t\mapsto\left(X,\,X+t(Y-X)\right),\qquad 0\le t\le1, \]
stays in $\mathcal K\times\mathcal K$.
Applying Taylor's formula with integral remainder along this segment, and using $N(X,X)=X$ and $D_2N_{(X,X)}=\frac12\operatorname{id}$, gives
\[ N(X,Y) = \frac{X + Y}{2} + R_N(X,Y), \]
where
\[ R_N(X,Y) \coloneqq \int_0^1 (1-t)\, D_2^2 N_{(X,\, X + t(Y-X))}\!\left[Y-X,\, Y-X\right] dt, \]
so that $\left| R_N(X,Y) \right| \le C_N\, |Y - X|^2$.
For both $A$ and $M$, the sum of the zeroth- and first-order terms is the coordinate arithmetic mean $\frac{X+Y}{2}$, so these terms cancel in the difference, and
\begin{equation}\label{eq:general-quadratic-gap}
 \left| A(X,Y) - M(X,Y) \right| = \left| R_A(X,Y) - R_M(X,Y) \right| \le C\, |X-Y|^2
\end{equation}
for every $(X,Y)\in\mathcal K\times\mathcal K$.
This cancellation follows from symmetry and idempotence through the diagonal derivatives computed above, and gives the quadratic rate.
By the same expansion, we obtain the following bound for the displacement of each argument in one step:
\begin{equation}\label{eq:drift}
 \left|N(X,Y)-X\right| \le \left(\frac{1}{2}+C_N|X-Y|\right)|X-Y| \le |X-Y|
\end{equation}
for every $(X,Y)\in\mathcal K\times\mathcal K$ with $\kappa|X-Y|\le\frac{1}{2}$.

Set $\mathcal U\coloneqq\mathcal B_{r_0/2}$ and
\[
\mathcal D\coloneqq\left\{(X,Y)\in\mathcal B_{r_0}\times\mathcal B_{r_0}\colon
\begin{array}{l}
|X-Y|<\varepsilon,\\
\max\{|X-o|,|Y-o|\}+2|X-Y|<\frac{3}{4}r_0
\end{array}\right\}.
\]
The set $\mathcal D$ is open and invariant under exchanging its two arguments.
Moreover, $\Omega_\varepsilon\subset\mathcal D$ because $\max\{|X-o|,|Y-o|\}+2|X-Y|<\frac{r_0}{2}+2\varepsilon\le\frac{3}{4}r_0$ for $(X,Y)\in\Omega_\varepsilon$.
We claim that $T(\mathcal D)\subset\mathcal D$.
Let $(X,Y)\in\mathcal D$ and $(X',Y')\coloneqq T(X,Y)$.
By \eqref{eq:general-quadratic-gap} and $C\varepsilon\le\frac12$,
\[ |X'-Y'|\le C|X-Y|^2\le\frac{1}{2}|X-Y|<\varepsilon. \]
By \eqref{eq:drift}, both $|X'-X|$ and $|Y'-X|$ are at most $|X-Y|$.
Hence
\[
\begin{aligned}
&\max\{|X'-o|,|Y'-o|\}+2|X'-Y'|\\
&\qquad\le |X-o|+2|X-Y|\\
&\qquad\le\max\{|X-o|,|Y-o|\}+2|X-Y|<\frac{3}{4}r_0.
\end{aligned}
\]
In particular, $X',Y'\in\mathcal B_{r_0}$, so $(X',Y')\in\mathcal D$.
This proves the claim.

Let $(P_0,Q_0)\in\mathcal D$.
By the claim, $(P_n,Q_n)\in\mathcal D\subset\mathcal K\times\mathcal K$ for every $n$, and \eqref{eq:general-quadratic-gap} with $C\delta_n\le C\varepsilon\le\frac12$ gives $\delta_{n+1}\le C\delta_n^2\le\frac12\delta_n$, which is (1).
In particular, $\delta_k\le2^{-(k-n)}\delta_n$ for $k\ge n$, and by \eqref{eq:drift},
\[ |P_m-P_n|\le\sum_{k=n}^{m-1}|P_{k+1}-P_k|\le\sum_{k\ge n}\delta_k\le2\delta_n, \qquad m\ge n. \]
Thus $(P_n)$ is a Cauchy sequence in the compact set $\mathcal K$; let $L(P_0,Q_0)\in\mathcal K$ be its limit.
Letting $m\to\infty$ gives $|P_n-L(P_0,Q_0)|\le2\delta_n$, and $|Q_n-L(P_0,Q_0)|\le\delta_n+|P_n-L(P_0,Q_0)|\le3\delta_n$.
Since $\delta_n\to0$, both sequences converge to $L(P_0,Q_0)$, which is (2).

Symmetry of $A$ and $M$ gives $T(Y,X)=T(X,Y)$, so $L(Y,X)=L(X,Y)$; idempotence gives $T(X,X)=(X,X)$, so $L(X,X)=X$.
Since $T$ maps $\mathcal D$ continuously into itself, each iterate $T^{n}$ is continuous on $\mathcal D$, and by (2),
\[ \left|P_n-L(P_0,Q_0)\right|\le2\delta_n\le2^{1-n}\varepsilon, \qquad (P_0,Q_0)\in\mathcal D, \]
so $L$ is the uniform limit on $\mathcal D$ of the continuous maps $(P_0,Q_0)\mapsto P_n$ and is therefore continuous.

For (3), the expansion of $A$ and $M$ gives, for every $n\ge0$,
\[ \frac{P_{n+1}+Q_{n+1}}{2}-\frac{P_n+Q_n}{2}=\frac{1}{2}\left(R_A+R_M\right)(P_n,Q_n), \]
whose norm is at most $\frac{C}{2}\delta_n^2$.
Since $\frac{P_1+Q_1}{2}=\frac{A(P_0,Q_0)+M(P_0,Q_0)}{2}$ and $\frac{P_n+Q_n}{2}\to L(P_0,Q_0)$, summing over $n\ge1$ and using $\delta_1\le C\delta_0^2$ and $\delta_n\le2^{-(n-1)}\delta_1$ yields
\[ \left|L(P_0,Q_0)-\frac{A(P_0,Q_0)+M(P_0,Q_0)}{2}\right|\le\frac{C}{2}\sum_{n\ge1}\delta_n^2\le\frac{C}{2}\cdot\frac{4}{3}\delta_1^2\le\frac{2}{3}C^{3}\delta_0^{4}. \qedhere \]
\end{proof}

Theorem~\ref{thm:local-existence} does not require a metric.
When $\mathcal P$ carries a Riemannian metric $g$, the coordinate norm and the Riemannian distance $d_g$ are uniformly comparable on the compact set $\mathcal K$.
Thus the quadratic gap estimate and the convergence bounds in (1) and (2) hold with $d_g$ in place of $|\cdot|$ after changing the constants.
We use this fact in the following theorem to identify the limit.
The form of estimate (3) depends on the coordinates, but the estimate holds in every chart with a constant depending on the chart.
It states that the Gauss composition differs from the coordinate arithmetic mean of its two components by $O(|X-Y|^4)$.
In general, the difference between two symmetric idempotent maps is only $O(|X-Y|^2)$.
In the one-dimensional scalar case with $A$ the arithmetic and $M$ the geometric mean, this is the classical observation that $\AGM(1-h,1+h)$ and $\frac{1}{2}\left(A+G\right)(1-h,1+h)$ differ by $O(h^{4})$, while $A-G$ is of order $h^{2}$.

We recall two classical notions; see \cite{whitehead-1932-convex} and \cite[Chapter~III]{kobayashi-nomizu-1963}.
A neighborhood $\mathcal V$ of a point $o \in \mathcal P$ is a \emph{normal neighborhood} of $o$ if $\mathcal V$ is the diffeomorphic image under $\operatorname{Exp}^{\nabla^{g}}_{o}$ of a star-shaped open neighborhood of the origin of $T_{o}\mathcal P$.
The neighborhood $\mathcal V$ is \emph{strongly convex} with respect to $g$ if any two points of $\mathcal V$ are joined by a unique minimizing $g$-geodesic and the geodesic segment is contained in $\mathcal V$; in particular, since subarcs of minimizing geodesics are minimizing, the affine midpoint of such a segment bisects the Riemannian distance between its endpoints.
Every point of $\mathcal P$ has arbitrarily small strongly convex normal neighborhoods \cite{whitehead-1932-convex}.
On a strongly convex normal neighborhood $\mathcal V$, the Levi--Civita midpoint $M_g$ is defined and continuous on $\mathcal V\times\mathcal V$, because the unique minimizing geodesic segment depends continuously on its endpoints \cite[Chapter~III]{kobayashi-nomizu-1963}.

\begin{Thm}[Local identification of the Gauss composition from midpoint invariance]\label{thm:local-convergence}
Write $d_g$ for the Riemannian distance induced by $g$.
Fix $o\in\mathcal P$, and let $\mathcal V$ be a strongly convex normal neighborhood of $o$ with respect to $g$.
Suppose that the local midpoint maps $A$ and $M$ are uniquely defined and of class $C^{2}$ on an open neighborhood $\mathcal W\subset\mathcal V\times\mathcal V$ of $(o,o)$, with $A(X,Y),M(X,Y)\in\mathcal V$ for every $(X,Y)\in\mathcal W$, and that
\begin{equation}\label{eq:local-midpoint-invariance}
 M_g\left(A(X,Y),M(X,Y)\right)=M_g(X,Y)
\end{equation}
holds for every $(X,Y)\in\mathcal W$.
Let $\mathcal D$ be the $T$-invariant set of Theorem~\ref{thm:local-existence} for a coordinate chart around $o$, chosen with $\mathcal K\subset\mathcal V$.
Then there exists a constant $C_g>0$ such that for every $(P_0,Q_0)\in\mathcal D$, with $\delta_n\coloneqq d_g(P_n,Q_n)$:\\
(1) $(A\otimes M)(P_0,Q_0)=M_g(P_0,Q_0)$;\\
(2) for every $n\ge0$,
\begin{equation}\label{eq:local-quadratic-gap}
\begin{gathered}
\delta_{n+1}\le C_g\delta_n^{2},\\
d_g\left(P_n,M_g(P_0,Q_0)\right)
=d_g\left(Q_n,M_g(P_0,Q_0)\right)=\frac{\delta_n}{2}.
\end{gathered}
\end{equation}
\end{Thm}

\begin{proof}
By Theorem~\ref{thm:local-existence}, the Gauss composition $L=A\otimes M$ is defined on $\mathcal D$, $T(\mathcal D)\subset\mathcal D$, and $L(\mathcal D)\subset\mathcal K\subset\mathcal V$.
The midpoint means $A$, $M$, and $M_g$ are defined on $\mathcal D\subset\mathcal W$, and $M_g$ is continuous at $\left(L(P_0,Q_0),L(P_0,Q_0)\right)\in\mathcal V\times\mathcal V$ for every $(P_0,Q_0)\in\mathcal D$.
Since \eqref{eq:local-midpoint-invariance} holds on $\mathcal D$, Theorem~\ref{thm:invariance} applied to the domain $\mathcal D$ gives $A\otimes M=M_g$ on $\mathcal D$, which is (1).

For (2), choose $c_1,c_2>0$ with $c_1|X-Y|\le d_g(X,Y)\le c_2|X-Y|$ for $X,Y\in\mathcal K$.
With $C$ as in Theorem~\ref{thm:local-existence}(1), the quadratic gap estimate gives
\[ \delta_{n+1}\le c_2|P_{n+1}-Q_{n+1}|\le c_2C|P_n-Q_n|^2\le\frac{c_2C}{c_1^{2}}\delta_n^{2} \]
for every $n\ge0$, so we may take $C_g\coloneqq c_2C/c_1^2$.
Write $L=M_g(P_0,Q_0)$.
By (1) and the invariance of the Gauss composition under $T$, $M_g(P_n,Q_n)=(A\otimes M)(P_n,Q_n)=L$ for every $n\ge0$.
Since $P_n,Q_n\in\mathcal V$ and the defining $g$-geodesic segment is minimizing, its midpoint bisects the distance between its endpoints.
Thus $d_g(P_n,L)=d_g(Q_n,L)=\delta_n/2$, which proves \eqref{eq:local-quadratic-gap}.
\end{proof}

\begin{Rem}[Gauss composition and the quartic term]\label{rem:local-theorems}
Estimate (3) of Theorem~\ref{thm:local-existence} relates the Gauss composition to the expansions of Section~\ref{sec:conjecture}.
That estimate alone does not identify the quartic term, since its remainder is itself of order four.
For the smooth Euclidean dual structures of Sections~\ref{sec:parallel-cubic} and~\ref{sec:conjecture}, write $\nabla=\nabla^E+\Gamma$ and $\nabla^*=\nabla^E-\Gamma$.
The quadratic terms of their midpoint maps in \eqref{eq:expansion4} cancel, so, locally and uniformly near the diagonal,
\[
\left|\frac{A(X,Y)+M(X,Y)}{2}-\frac{X+Y}{2}\right|\le B|X-Y|^4
\]
for some $B>0$.
Put $c_n=(P_n+Q_n)/2$ and $\delta_n=|P_n-Q_n|$, and choose the sets $\mathcal K$ and $\mathcal D$ in Theorem~\ref{thm:local-existence} sufficiently small that this estimate holds on $\mathcal K\times\mathcal K$.
Then $|c_{n+1}-c_n|\le B\delta_n^4$, and the estimates of that theorem give
\[
|L(P_0,Q_0)-c_1|
\le B\sum_{n\ge1}\delta_n^4
\le\frac{16}{15}B\delta_1^4
\le\frac{16}{15}BC^4\delta_0^8.
\]
Consequently,
\[
(A\otimes M)(P,Q)-M_g(P,Q)
=\frac{A(P,Q)+M(P,Q)-(P+Q)}{2}+O(|P-Q|^8).
\]
Thus the quartic term of $A\otimes M-M_g$ is one half of that of $A+M-(P+Q)$.
In the notation of Section~\ref{sec:conjecture}, \eqref{eq:expansion4} yields
\[
(A\otimes M)\left(o-\frac h2,o+\frac h2\right)-o
=\frac{1}{384}\mathcal E_\Gamma(o;h)+O(|h|^6).
\]
When the invariance identity $A+M=P+Q$ holds, $A\otimes M=(A+M)/2=M_g$ exactly.
\end{Rem}

\begin{Rem}[Comparison with the Gauss composition of quasi-arithmetic means]\label{rem:qam-comparison}
For quasi-arithmetic means on an interval, conjugation by the generator of one component makes that component arithmetic; see \cite{daroczy-pales-2002,jarczyk-jarczyk-2018}.
If the invariant mean is also quasi-arithmetic, it can instead be made arithmetic by conjugation with its generator, as in Remark~\ref{rem:suto}.
These two normalizations need not coincide.
For a torsion-free connection $\nabla$, a local coordinate representation making $A$ arithmetic exists exactly when $\nabla$ is flat.
Indeed, a smooth midpoint map determines the torsion-free connection through its second-order expansion at the diagonal, and the converse follows from affine coordinates (Section~\ref{sec:setting}).
Thus the non-flat examples in Sections~\ref{subsec:alpha} and~\ref{sec:parallel-cubic} are not covered by the quasi-arithmetic theory.
\end{Rem}

\section{A point-symmetry criterion for invariance}\label{sec:symmetry}

In this section, we give a sufficient condition for the invariance identity \eqref{eq:Mg-inv} of Theorem~\ref{thm:invariance} in arbitrary dimension.
The condition is based on the symmetry used in Nakamura's invariance identity, as we show in Section~\ref{sec:nakamura}.
It also gives this identity for a whole one-parameter affine family of dual connections.
In Section~\ref{sec:one-dim}, we give necessary conditions by restricting to one-dimensional submanifolds. 

Suppose that the standing assumptions of Section~\ref{sec:setting} are valid for the pairs $(P,Q)$ and $\left(A(P,Q),M(P,Q)\right)$. 
Specifically, we assume that $(P,Q)$ is joined by a unique geodesic of each of $\nabla$, $\nabla^{*}$, and $\nabla^{g}$, and $\left(A(P,Q),M(P,Q)\right)$ is joined by a unique $\nabla^{g}$-geodesic, so that the midpoint means $A(P,Q)$, $M(P,Q)$, $M_g(P,Q)$, and $M_g\left(A(P,Q),M(P,Q)\right)$ appearing in the invariance identity \eqref{eq:Mg-inv} are defined (Definition~\ref{def:midpoint-mean}).

For a diffeomorphism $s$ and an affine connection $D$, 
we use the push-forward convention $(s_*D)_{s_*X}(s_*Y) \coloneqq s_*(D_{X} Y)$ for vector fields $X,Y$. 
With this convention, $s_*D = D'$ means that $s$ maps affinely parametrized $D$-geodesics to affinely parametrized $D'$-geodesics. 
When $s(U)=U$, the same notation is used for the corresponding restricted connections on $U$.

\begin{Thm}[Point-symmetry criterion]\label{thm:symmetry}
Let $(\mathcal{P},g)$ be a Riemannian manifold equipped with an affine connection $\nabla$ and its $g$-dual connection $\nabla^*$.
Fix a pair $(P,Q)$ and set $o\coloneqq M_g(P,Q)$.
Assume that the midpoint constructions $A(P,Q), M(P,Q), M_g(P,Q)$ and $M_g\left(A(P,Q),M(P,Q)\right)$ are uniquely defined.
Suppose that there exist an isometry $s$ of $(\mathcal{P},g)$ and an $s$-invariant domain $\mathcal{U} \subset \mathcal{P}$ such that $P, Q \in \mathcal{U}$ and all the geodesic segments defining the above midpoints are contained in $\mathcal{U}$. 
Assume further that\\
(a) $s(o) = o$ and  $ds_o = -\operatorname{id}_{T_o\mathcal{P}}$;\\
(b) $s_{*} \nabla = \nabla^*$ on $\mathcal{U}$;\\
(c) $\operatorname{Fix}(s)\cap \mathcal{U} = \{o\}$, where we let $\operatorname{Fix}(s) \coloneqq \left\{X \in \mathcal{P} \colon s(X) = X \right\}$.\\
Then 
\[ M_g \left(A(P,Q),M(P,Q)\right) = M_g (P,Q).\] 
\end{Thm}

\begin{proof}
Since $o$ is the midpoint of the unique $g$-geodesic joining $P$ and $Q$, we may write $P = \operatorname{Exp}^{\nabla^{g}}_{o}(-v)$ and $Q = \operatorname{Exp}^{\nabla^{g}}_{o}(v)$. 
An isometry maps affinely parametrized geodesics to affinely parametrized geodesics. 
Therefore (a) gives $s\left(\operatorname{Exp}^{\nabla^{g}}_{o}(u)\right) = \operatorname{Exp}^{\nabla^{g}}_{o}(-u)$. 
In particular, $s$ exchanges $P$ and $Q$.

Because $s$ preserves $g$, push-forward by $s$ commutes with $g$-duality:
\[
(s_{*}D)^{*} = s_{*}(D^{*}) \qquad \text{for every affine connection } D.
\]
This follows by applying the isometry $s$ to the defining identity \eqref{eq:duality} and using $s^{*}g = g$. 
Hence (b) implies
\[ s_{*}\nabla^{*} = (s_{*}\nabla)^{*} = (\nabla^{*})^{*} = \nabla. \]

For any affine connection $D$ and any pair of points $X,Y$ for which the indicated midpoint is uniquely defined, the push-forward convention gives
\[ s\left(M_D(X,Y)\right) = M_{s_*D}\left(s(X),s(Y)\right). \]
Indeed, $s$ maps the unique affinely parametrized $D$-geodesic joining $X$ and $Y$ to the corresponding $s_*D$-geodesic joining $s(X)$ and $s(Y)$, preserving its affine parametrization.

Applying this to $D = \nabla$ and $D = \nabla^{*}$, and using that midpoint means are symmetric in their two arguments,
\[ s(A(P,Q)) = M_{\nabla^{*}}(Q,P) = M(P,Q), \ \ s(M(P,Q)) = M_{\nabla}(Q,P) = A(P,Q), \]
so $s$ exchanges $A(P,Q)$ and $M(P,Q)$.

Let $m \coloneqq M_g\left(A(P,Q),M(P,Q)\right)$. 
By the choice of $\mathcal{U}$, we have $m \in \mathcal{U}$. 
Since $s$ is an isometry, it preserves the Levi--Civita connection and hence its midpoint map. 
Therefore,
\[ s(m) = M_g\left(s(A(P,Q)),s(M(P,Q))\right) = M_g \left(M(P,Q),A(P,Q)\right) = m. \]
Thus $m \in \operatorname{Fix}(s)\cap \mathcal{U}$. 
By condition~(c), $m=o$.
Since $o = M_g (P,Q)$, the desired invariance identity follows.
\end{proof}

Theorem~\ref{thm:symmetry} applies to each fixed pair $(P,Q)$.
If its hypotheses hold for every pair in a neighborhood of the diagonal, then the midpoint invariance holds on that neighborhood.
By Theorem~\ref{thm:local-convergence}, the mean-type iteration then converges locally and quadratically to $M_g(P,Q)$.

\begin{Rem}\label{rem:symmetry}
(1) 
By the Cartan--Hadamard theorem, a Hadamard manifold is uniquely geodesic. 
Hence condition~(c) is automatic under~(a). 
Indeed, suppose that $s(X) = X$ with $X \ne o$.
Then $s$ maps the unique geodesic from $o$ to $X$ to a geodesic from $o$ to $X$.
By uniqueness, it maps this geodesic to itself with the same parametrization.
Thus $ds_{o}$ fixes its initial velocity, which contradicts (a). 
The same argument applies locally on any $s$-invariant convex normal neighborhood of $o$.\\
(2) On a Riemannian symmetric space, the geodesic symmetry $s_o$ satisfies (a).
By (1), condition (c) holds locally, but may require a separate check globally.\\
(3) Assume additionally that both $\nabla$ and $\nabla^*$ are torsion-free. 
Let 
\[ K \coloneqq \nabla - \nabla^g, \quad \nabla^{(\alpha)}\coloneqq\nabla^g + \alpha K, \quad \alpha \in \mathbb{R}. \]
Then
\[ \nabla^{(1)} = \nabla,\qquad \nabla^{(-1)} = \nabla^*,\qquad \left(\nabla^{(\alpha)}\right)^{*}   =\nabla^{(-\alpha)}. \]
Here $\nabla^{(-1)} = \nabla^{*}$ is \eqref{eq:dual-average}.
If an isometry $s$ satisfies $s_{*} \nabla=\nabla^*$, then,
since $s_{*} \nabla^g = \nabla^g$,
$s_{*} K = -K$ and hence $s_{*} \nabla^{(\alpha)} = \nabla^{(-\alpha)}$ for every $\alpha \in \mathbb{R}$.

Let $A_\alpha\coloneqq M_{\nabla^{(\alpha)}}$ and $M_\alpha\coloneqq M_{\nabla^{(-\alpha)}}$.
Suppose that for each $(P,Q)\in\Omega\subset\mathcal P\times\mathcal P$ there is an isometry $s_{P,Q}$ satisfying (a) at $o=M_g(P,Q)$ and exchanging $\nabla$ and $\nabla^*$.
For every $\alpha$ under consideration, assume that all geodesic segments required in Theorem~\ref{thm:symmetry} for the dual pair $(\nabla^{(\alpha)},\nabla^{(-\alpha)})$ are uniquely defined and lie in an $s_{P,Q}$-invariant domain $\mathcal U_{P,Q,\alpha}$ with $\operatorname{Fix}(s_{P,Q})\cap\mathcal U_{P,Q,\alpha}=\{o\}$.
Then Theorem~\ref{thm:symmetry} gives $M_g\left(A_\alpha(P,Q),M_\alpha(P,Q)\right)=M_g(P,Q)$ on $\Omega$.
For fixed $\alpha$, put $T_\alpha=(A_\alpha,M_\alpha)$.
If $T_\alpha(\Omega)\subset\Omega$, the iterates converge to a common limit for every pair in $\Omega$, and $M_g$ is defined and continuous at each limiting diagonal pair, then Theorem~\ref{thm:invariance} gives $A_\alpha\otimes M_\alpha=M_g$ on $\Omega$.
\end{Rem}

\section{Nakamura's affine-invariant example}\label{sec:nakamura}

On the manifold $\SPD(m)$ of $m\times m$ symmetric positive-definite matrices, 
consider the affine-invariant Riemannian metric
$g_P(X,Y) = \operatorname{tr}\left(P^{-1}XP^{-1}Y \right)$. 
For the Euclidean flat connection $\nabla^E$, the midpoint mean is the arithmetic mean $A(P,Q) = \frac{P + Q}{2}$.
The $g$-dual connection has geodesics that are affine in the inverse-matrix coordinate, and its midpoint mean is the harmonic mean $H(P,Q) = 2(P^{-1} + Q^{-1})^{-1}$.
The midpoint mean of the Levi--Civita connection is the affine-invariant geometric mean~\cite{nakamura-2001}
\[ G(P,Q) = P\#Q \coloneqq P^{1/2} \left(P^{-1/2}QP^{-1/2}\right)^{1/2} P^{1/2}. \]
With respect to this metric, $\SPD(m)$ is a Hadamard manifold and a Riemannian symmetric space~\cite{bhatia-2007}.

\subsection{Nakamura's arithmetic--harmonic iteration}\label{subsec:AHM}

We apply Theorem~\ref{thm:symmetry} to obtain the invariance identity for this structure without directly computing the matrix means.
Fix a pair $(P,Q)$ and set $o \coloneqq M_g(P,Q) = P\#Q$.
The geodesic symmetry at $o$ is $\sigma_o(X) \coloneqq oX^{-1}o$.
It satisfies
\[ \sigma_{o} (o) = o, \quad d(\sigma_o)_{o} = -\operatorname{id}_{T_o\SPD(m)}, \]
so condition~(a) of Theorem~\ref{thm:symmetry} holds.
Moreover, the standard Riccati characterization of the geometric mean gives $oP^{-1}o = Q$, and hence $\sigma_{o}(P) = Q$.

Write
\[ \sigma_o = c_o\circ\iota, \quad \iota(X) \coloneqq X^{-1}, \quad c_o(Y) \coloneqq oYo. \]
Both $\iota$ and $c_o$ are isometries of the affine-invariant metric. 
The inversion map exchanges the Euclidean connection
and its $g$-dual: $\iota_*\nabla^E=(\nabla^E)^{*}$.
Since $c_o$ is linear, it preserves $\nabla^E$.
Since it is also an isometry, it preserves the dual connection $(\nabla^E)^*$.
Consequently,
\[ (\sigma_o)_*\nabla^E = (c_o)_*\left(\iota_*\nabla^E\right) = (c_o)_*(\nabla^E)^* =(\nabla^E)^*. \]
Thus condition~(b) of Theorem~\ref{thm:symmetry} holds; condition~(c) holds by Remark~\ref{rem:symmetry}(1), since $\SPD(m)$ is a Hadamard manifold.
Theorem~\ref{thm:symmetry} yields
\[ M_g\left(A(P,Q),H(P,Q)\right) = M_g(P,Q), \]
which is the one-step midpoint invariance in Nakamura's matrix arithmetic--harmonic iteration.
Its scalar case, the invariance of the geometric mean under the arithmetic--harmonic pair, was recorded by Sut\^{o} in 1914 as an instance of his general solution of the invariance problem for quasi-arithmetic means; see Remark~\ref{rem:suto}.

The invariance requirement also determines the harmonic mean.
Indeed, if a mean $N$ satisfies
$A(P,Q)\#N(P,Q)  =  P\#Q$, 
then the equation $X\#Y  =  Z$ gives $Y  =  ZX^{-1}Z$, and hence
\[ N(P,Q)  =  (P\#Q)A(P,Q)^{-1}(P\#Q)  =  2(P^{-1} + Q^{-1})^{-1}  =  H(P,Q). \]
Thus, for the affine-invariant metric and the Euclidean connection, the dual midpoint mean is $H$.
It is also the unique complementary mean under which $G$ is invariant.

Nakamura proved by a linear algebraic argument that the matrix arithmetic--harmonic iteration converges quadratically for every initial pair.
He first treated the case $Q_0 = I$ using the identity $Q_{n+1} - R_{n+1} = \frac{1}{2}(Q_n - R_n)^2(Q_n + R_n)^{-1}$, and then obtained the general case by congruence \cite[Theorems~9 and~10]{nakamura-2001}.
Together with this convergence, the invariance identity gives Nakamura's identity $A \otimes H  =  G$ by Theorem~\ref{thm:invariance}.
For sufficiently close pairs, Theorem~\ref{thm:local-convergence} gives quadratic convergence from the invariance identity alone.
Thus it follows from the symmetry of Theorem~\ref{thm:symmetry}, without using the explicit matrix identities.

\subsection{A one-parameter family of power-mean iterations}\label{subsec:alpha}

By Remark~\ref{rem:symmetry}(3), the isometry $\sigma_o$ also exchanges the connections
\[ \nabla^{(\alpha)} \coloneqq \nabla^g + \alpha\left(\nabla^E-\nabla^g\right) \quad \text{and} \quad \nabla^{(-\alpha)}, \qquad \alpha \in \mathbb{R}, \]
and we now apply Theorem~\ref{thm:symmetry} to the pair $\left(\nabla^{(\alpha)},\nabla^{(-\alpha)}\right)$ with $\mathcal{U} = \SPD(m)$.
Conditions (a) and (c) hold as in Section~\ref{subsec:AHM}, and it remains to verify that any two points of $\SPD(m)$ are joined by a unique $\nabla^{(\alpha)}$-geodesic segment.
In the ambient matrix coordinates,
\[ \left(\nabla^{(\alpha)}_X Y\right)_P = DY(P)[X] - \frac{1-\alpha}{2}  \left(XP^{-1}Y + YP^{-1}X\right).\]
For $-1\le\alpha\le1$, these connections and their power-mean geodesics were studied by P\'alfia~\cite[Theorems~7.1 and~12.4]{palfia-2013}, with the parameter correspondence $\kappa=1-\alpha$.
We recall the calculation for all real $\alpha$ to verify the hypotheses of Theorem~\ref{thm:symmetry} on the whole cone.
The geodesic equation is
\[ \ddot{X} -(1-\alpha) \dot{X} X^{-1} \dot{X} = 0.\]
Let $\alpha \ne 0$.
The $\nabla^{(\alpha)}$-geodesic with $\gamma(0) = P$ and $\dot{\gamma}(0) = V$ is
\begin{equation}\label{eq:alpha-geodesic}
\gamma(t) = P^{1/2}\left(I + t\alpha W\right)^{1/\alpha}P^{1/2}, \qquad W \coloneqq P^{-1/2}VP^{-1/2},
\end{equation}
on the interval on which $I + t\alpha W$ is positive definite, all matrix powers being the principal positive-definite powers.
Indeed, the geodesic equation is invariant under congruences $X \mapsto CXC^{\top}$, so it suffices to take $P = I$.
Then $Z(t) \coloneqq I + t\alpha W$ satisfies $\ddot{Z} = 0$.
All the matrices involved are functions of $W$ and commute.
Hence $Y \coloneqq Z^{1/\alpha}$ satisfies $\dot{Y} = Z^{1/\alpha - 1}W$ and $\ddot{Y} = (1-\alpha)Z^{1/\alpha - 2}W^{2} = (1-\alpha)\dot{Y}Y^{-1}\dot{Y}$.
Every $\nabla^{(\alpha)}$-geodesic segment in $\SPD(m)$ starting at $P$ is of the form \eqref{eq:alpha-geodesic}, and $\gamma(1) = Q$ holds if and only if $W = \alpha^{-1}\left(S^{\alpha} - I\right)$, where
\[ S \coloneqq P^{-1/2} Q P^{-1/2}. \]
Hence the unique $\nabla^{(\alpha)}$-geodesic segment from $P$ to $Q$ is
\[  \gamma_\alpha(t) = P^{1/2}\left((1-t)I+tS^\alpha\right)^{1/\alpha}P^{1/2}, \quad 0 \le t \le 1, \]
with the limiting interpretation $\gamma_0(t) = P^{1/2} S^t P^{1/2}$ for $\alpha = 0$.
Thus the standing assumptions of Section~\ref{sec:setting} hold for every $\alpha$, and Theorem~\ref{thm:symmetry} yields
\begin{equation}\label{eq:alpha-invariance}
M_g\left(A_\alpha(P,Q),M_\alpha(P,Q)\right) = M_g(P,Q), \quad P, Q \in \SPD(m), \quad \alpha \in \mathbb{R},
\end{equation}
where $A_\alpha \coloneqq M_{\nabla^{(\alpha)}}$ and $M_\alpha \coloneqq M_{\nabla^{(-\alpha)}}$.

We now give explicit formulas for the midpoint means.
Let
\[  f_\alpha(S) \coloneqq 
  \begin{cases}
    \displaystyle
    \left(\frac{I+S^\alpha}{2}\right)^{1/\alpha}, & \alpha \ne 0,\\ 
    S^{1/2},&\alpha = 0.
  \end{cases}\]
Then, by $\gamma_\alpha(1/2)$,
\begin{equation}\label{eq:alpha-means}
A_{\alpha} (P,Q) = P^{1/2} f_\alpha(S)P^{1/2}, \qquad M_{\alpha} (P,Q) = P^{1/2} f_{-\alpha}(S)P^{1/2}. 
\end{equation}
For $-1\le\alpha\le1$, $A_\alpha$ is the Kubo--Ando mean represented by $f_\alpha$~\cite{kubo-ando-1980}; this function is operator monotone precisely in that parameter range~\cite[Example~4.1]{palfia-2013}.
It is also the two-variable, equal-weight power mean of Lim and P\'alfia; see Remark~\ref{rem:operator-power-means} for its fixed-point characterization and its relation to Kim's notation.
For $|\alpha|>1$, the midpoint maps remain well-defined, but the family is outside the class of Kubo--Ando operator means.
When $P$ and $Q$ do not commute, the surrounding congruence factors cannot in general be omitted.
The quasi-arithmetic matrix power means $\big(\tfrac{1}{2}(P^{\alpha} + Q^{\alpha})\big)^{1/\alpha}$ and $\big(\tfrac{1}{2}(P^{-\alpha} + Q^{-\alpha})\big)^{-1/\alpha}$ do not in general coincide with the midpoint means of $\nabla^{(\pm\alpha)}$.
Their pair does not in general leave $P\#Q$ invariant either (Remark~\ref{rem:power-pullback-metrics}).
Identity~\eqref{eq:alpha-invariance} can also be verified from \eqref{eq:alpha-means}: since $f_{-\alpha}(S) = S f_\alpha(S)^{-1}$, the matrices $f_\alpha(S)$ and $f_{-\alpha}(S)$ commute, and
\[ A_\alpha(P,Q)\#M_\alpha(P,Q) = P^{1/2}\big(f_\alpha(S) f_{-\alpha}(S)\big)^{1/2}P^{1/2} = P^{1/2}S^{1/2}P^{1/2}, \]
which is $P\#Q$, in accordance with \eqref{eq:alpha-invariance}.

In particular, $(A_1,M_1) = (A,H)$, $A_0 = M_0 = M_g$, and $(A_{-1},M_{-1}) = (H,A)$.
For $m\ge2$ and $\alpha\ne\pm1$, the connection $\nabla^{(\alpha)}$ is not flat, since the $\alpha$-connections of a dually flat structure satisfy $R^{\nabla^{(\alpha)}} = (1-\alpha^{2})R^{\nabla^{g}}$~\cite{AmariNagaoka2000} and $R^{\nabla^g} \ne 0$; hence, unlike $A$ and $H$, the corresponding means $A_\alpha$ do not admit an affine-coordinate quasi-arithmetic representation (Section~\ref{sec:setting}).

We next give the resulting one-parameter family of Gauss compositions.
Since the midpoint maps \eqref{eq:alpha-means} are smooth, Theorem~\ref{thm:local-convergence} gives quadratic convergence for sufficiently close initial pairs.
In this example, the convergence is global.
We show this by reducing the iteration to scalar iterations on the eigenvalues of $S$.

\begin{Prop}[A one-parameter family of Gauss compositions with limit $P\#Q$]\label{prop:alpha-family}
Let $\alpha \in \mathbb{R}$, let $P, Q \in \SPD(m)$, put $S \coloneqq P^{-1/2}QP^{-1/2}$, and define $(P_n,Q_n)_{n \ge 0}$ by $(P_0,Q_0) = (P,Q)$ and
\[ (P_{n+1},Q_{n+1}) = \left(A_\alpha(P_n,Q_n), M_\alpha(P_n,Q_n)\right). \]
Let $a_n, m_n \colon (0,\infty) \to (0,\infty)$ be defined by $a_0(\lambda) = 1$, $m_0(\lambda) = \lambda$, and
\[ a_{n+1} = \left(\frac{a_n^{\alpha} + m_n^{\alpha}}{2}\right)^{1/\alpha}, \qquad m_{n+1} = \left(\frac{a_n^{-\alpha} + m_n^{-\alpha}}{2}\right)^{-1/\alpha} \]
for $\alpha \ne 0$, and by $a_{n+1} = m_{n+1} = \sqrt{a_n m_n}$ for $\alpha = 0$.
Let 
\[ \delta_n \coloneqq \frac{1}{2}\max_{1 \le i \le m} \left|\log \lambda_i\left(P_n^{-1}Q_n\right)\right|,\] 
where $\lambda_1(X), \dots, \lambda_m(X)$ denote the eigenvalues of $X$.
Then \\
(1) $P_n\#Q_n = P\#Q$ for all $n \ge 0$; \\
(2) $P_n = P^{1/2}a_n(S)P^{1/2}$ and $Q_n = P^{1/2}m_n(S)P^{1/2}$ for all $n \ge 0$; \\
(3) $\delta_{n+1} \le \tfrac{|\alpha|}{2}\,\delta_n^{2}$ for all $n \ge 0$, and $\delta_{n+1} < \delta_n$ whenever $\delta_n > 0$; \\
(4) $P_n \to P\#Q$ and $Q_n \to P\#Q$ as $n \to \infty$, and consequently $A_\alpha \otimes M_\alpha = M_g$ on $\SPD(m) \times \SPD(m)$.
\end{Prop}

\begin{proof}
(1) follows from \eqref{eq:alpha-invariance} by induction on $n$.

(2) The geodesic equation of $\nabla^{(\alpha)}$ and the formula \eqref{eq:alpha-geodesic} are invariant under congruences $X \mapsto CXC^{\top}$ with $C \in \mathrm{GL}(m,\mathbb{R})$, so that
\[ N\left(CXC^{\top},CYC^{\top}\right) = C\, N(X,Y)\, C^{\top} \qquad \text{for } N \in \{A_\alpha, M_\alpha\}. \]
If $X$ and $Y$ commute, then $X^{-1/2}YX^{-1/2} = X^{-1}Y$, and \eqref{eq:alpha-means} gives, for $\alpha \ne 0$,
\begin{align*}
A_\alpha(X,Y) &= X^{1/2}\left(\frac{I + X^{-\alpha}Y^{\alpha}}{2}\right)^{1/\alpha}X^{1/2} = \left(\frac{X^{\alpha} + Y^{\alpha}}{2}\right)^{1/\alpha}, \\
M_\alpha(X,Y) &= \left(\frac{X^{-\alpha} + Y^{-\alpha}}{2}\right)^{-1/\alpha},
\end{align*}
and $A_0(X,Y) = M_0(X,Y) = (XY)^{1/2}$.
Now (2) holds for $n = 0$, and if it holds for $n$, then $a_n(S)$ and $m_n(S)$ commute, so the displayed formulas with $C = P^{1/2}$, $X = a_n(S)$, and $Y = m_n(S)$ give $P_{n+1} = P^{1/2}a_{n+1}(S)P^{1/2}$ and $Q_{n+1} = P^{1/2}m_{n+1}(S)P^{1/2}$.

(3) For $\alpha \ne 0$, we have $a_{n+1}m_{n+1} = a_n m_n$, because
\[ \left(a^{\alpha} + m^{\alpha}\right)\left(a^{-\alpha} + m^{-\alpha}\right)^{-1} = a^{\alpha}m^{\alpha}, \qquad a,m>0. \]
The same identity $a_{n+1}m_{n+1} = a_n m_n$ also holds for $\alpha = 0$.
Hence $a_n(\lambda)m_n(\lambda) = \lambda$, and we may write $a_n = \sqrt{\lambda}\,e^{x_n}$, $m_n = \sqrt{\lambda}\,e^{-x_n}$ with $x_n = x_n(\lambda) \in \mathbb{R}$.
For $\alpha\ne0$, we have $a_{n+1} = \sqrt{\lambda}\,\cosh(\alpha x_n)^{1/\alpha}$, so
\begin{equation}\label{eq:alpha-scalar-recursion}
x_{n+1} = \frac{1}{\alpha}\log\cosh(\alpha x_n) \quad (\alpha \ne 0), \qquad x_{n+1} = 0 \quad (\alpha = 0).
\end{equation}
Since $\log\cosh s \le s^{2}/2$ and $\log\cosh s < |s|$ for $s \ne 0$, we obtain $|x_{n+1}| \le \tfrac{|\alpha|}{2}x_n^{2}$ and $|x_{n+1}| < |x_n|$ for $x_n \ne 0$.
By (2), the eigenvalues of $P_n^{-1}Q_n$ are $m_n(\lambda_i)/a_n(\lambda_i) = e^{-2x_n(\lambda_i)}$, where $\lambda_i = \lambda_i(S)$, so $\delta_n = \max\nolimits_i |x_n(\lambda_i)|$.
For $\alpha\ne0$, monotonicity of $t \mapsto |\alpha|^{-1}\log\cosh(|\alpha|t)$ on $[0,\infty)$ gives the exact recursion
\begin{equation}\label{eq:alpha-gap-recursion}
\delta_{n+1}=\frac{1}{|\alpha|}\log\cosh(|\alpha|\delta_n).
\end{equation}
Together with the preceding inequalities, this proves (3); for $\alpha=0$, it follows directly from $x_{n+1}=0$.

(4) The case $\alpha=0$ is immediate. Assume $\alpha\ne0$.
By (3), $(\delta_n)$ is nonincreasing.
Suppose that its limit $\delta$ is positive.
Letting $n \to \infty$ in \eqref{eq:alpha-gap-recursion}, we obtain $\delta = |\alpha|^{-1}\log\cosh(|\alpha|\delta) < \delta$, a contradiction.
Hence $\delta_n \to 0$, so $a_n(\lambda_i) \to \sqrt{\lambda_i}$ and $m_n(\lambda_i) \to \sqrt{\lambda_i}$ for every eigenvalue $\lambda_i$ of $S$, and (2) gives $P_n, Q_n \to P^{1/2}S^{1/2}P^{1/2} = P\#Q$.
\end{proof}

Hereafter, artanh denotes the inverse function of the hyperbolic tangent function. 

\begin{Cor}[Closed form of the scalar recursion and the iteration count]\label{cor:alpha-closed-form}
Let $\alpha \ne 0$, let $P, Q \in \SPD(m)$ with $P \ne Q$, let $\delta_n$ be as in Proposition~\ref{prop:alpha-family}, and put $\theta \coloneqq \operatorname{artanh}\left(e^{-|\alpha|\delta_0}\right)$. 
Then \\
(1) $\delta_n = |\alpha|^{-1}\log\coth\left(2^{n}\theta\right)$ for every $n \ge 0$. \\
(2) $\min\{n \ge 0 : \delta_n \le \varepsilon\} = \left\lceil \log_2 \dfrac{\operatorname{artanh}(e^{-|\alpha|\varepsilon})}{\theta} \right\rceil$ for every $\varepsilon \in (0,\delta_0)$.
\end{Cor}

\begin{proof}
Put $\beta \coloneqq |\alpha|$ and $y_n \coloneqq \beta\delta_n$.
By \eqref{eq:alpha-gap-recursion}, $y_{n+1}=\log\cosh(y_n)$.
Since $P\ne Q$, we have $\delta_0>0$, so $\theta=\operatorname{artanh}(e^{-\beta\delta_0})$ belongs to $(0,\infty)$ and $y_0=\log\coth\theta$.
The duplication formula $\coth(2y)=\tfrac12(\coth y+\tanh y)$ gives $\cosh(\log\coth y)=\coth(2y)$ for $y>0$.
Induction therefore yields $y_n=\log\coth(2^n\theta)$ for every $n\ge0$, which proves (1).
By (1), $\delta_n \le \varepsilon$ is equivalent to $\tanh(2^{n}\theta) \ge e^{-\beta\varepsilon}$, that is, to $2^{n}\theta \ge \operatorname{artanh}(e^{-\beta\varepsilon})$.
The right-hand side is greater than $\theta$ because $\varepsilon < \delta_0$.
This proves (2).
\end{proof}

For $\alpha = 1$, Proposition~\ref{prop:alpha-family}(4) is Nakamura's theorem.
The reduction in (2) to commuting matrices is the argument used in his proof for $Q_0 = I$.
For $\alpha = 0$, the iteration is stationary after one step.
Proposition~\ref{prop:alpha-family} gives a stronger conclusion than Theorem~\ref{thm:local-convergence} in this example.
By (3), the convergence is quadratic from every initial pair, with the explicit constant $|\alpha|/2$ in terms of $\delta_n$.
The latter is half the largest absolute logarithm of the generalized eigenvalues of the current pair.

\begin{Rem}[Dependence on $\alpha$ and on the eigenvalues]\label{rem:alpha-numerics}
By Proposition~\ref{prop:alpha-family}, the gap is determined by $\beta=|\alpha|$ and the extreme eigenvalues of $P^{-1}Q$ through $\delta_0$.
For $\beta>0$, it satisfies $\delta_{n+1}=h_\beta(\delta_n)$, where $h_\beta(t)=\beta^{-1}\log\cosh(\beta t)$.
The power-mean inequality shows that $h_\beta(t)$ is nondecreasing in $\beta$ for each $t\ge0$.
Since $h_\beta$ is also nondecreasing in $t$, $\delta_n$ is nondecreasing in $|\alpha|$ for every $n$.
Thus smaller $|\alpha|$ gives faster convergence of the gap, and $\alpha=0$ gives $M_g$ in one step.

This comparison does not account for the cost of each step.
For $\alpha=\pm1$, only matrix inversions and additions are needed.
For general $\alpha\ne0,\pm1$, evaluating the matrix powers in \eqref{eq:alpha-means} typically uses a spectral decomposition of $S$, from which $P\#Q$ can be obtained directly.
\end{Rem}

\begin{Rem}[Operator power means and adjoints]\label{rem:operator-power-means}
Let $\mu=\tfrac12(\delta_P+\delta_Q)$.
In the notation of Kim~\cite[Definition~4.1 and Theorem~4.3]{kim-2018},
\[
A_\alpha(P,Q)=P_\alpha(\mu),\qquad
M_\alpha(P,Q)=P_{-\alpha}(\mu),\qquad -1\le\alpha\le1,
\]
where $P_\alpha$ denotes the power mean of Lim and P\'alfia~\cite{lim-palfia-2012}, extended to compactly supported probability measures, and $P_0=\Lambda$ is the Cartan barycenter.
Indeed, for $0<\alpha\le1$, $P_\alpha(\mu)$ is the unique positive-definite solution of
\[
X=\tfrac12(X\#_\alpha P+X\#_\alpha Q),\qquad
X\#_\alpha Z\coloneqq X^{1/2}(X^{-1/2}ZX^{-1/2})^\alpha X^{1/2}.
\]
It suffices by congruence equivariance to substitute $X=f_\alpha(S)$ for $(P,Q)=(I,S)$.
The negative parameters follow from inversion duality, and at $\alpha=0$ both sides equal $\Lambda(\mu)=P\#Q$.

For a binary mean $N$, define its adjoint by
$N^*(P,Q)\coloneqq[N(P^{-1},Q^{-1})]^{-1}$.
The identity $f_\alpha(s)=s f_\alpha(1/s)$ gives
\[
f_\alpha^*(s)\coloneqq\frac{1}{f_\alpha(1/s)}
=\frac{s}{f_\alpha(s)}=f_{-\alpha}(s).
\]
Thus $M_\alpha=A_\alpha^*$.
Here the metric duality of the connections corresponds to the adjoint operation on means.
This identity holds for every real $\alpha$.
More generally, if $N$ is any symmetric Kubo--Ando mean with representing function $f$, then $f^*(s)=s/f(s)$, so the same commuting functional-calculus calculation as above yields
\[ N(P,Q)\#N^*(P,Q)=P\#Q. \]
Theorem~\ref{thm:symmetry} gives a geometric explanation of this identity for the power-mean connections.
The limit $P_\alpha(\mu)\to\Lambda(\mu)$ as $\alpha \to 0$ in~\cite{lim-palfia-2012,kim-2018} concerns variation of $\alpha$.
In Proposition~\ref{prop:alpha-family}, we iterate $(A_\alpha,A_{-\alpha})$ for a fixed $\alpha$.
\end{Rem}

\begin{Rem}[Quasi-arithmetic power means and pullback metrics]\label{rem:power-pullback-metrics}
For $\mu=\tfrac12(\delta_P+\delta_Q)$, Kim's quasi-arithmetic means are
\[
Q_\alpha(\mu)=\left(\frac{P^\alpha+Q^\alpha}{2}\right)^{1/\alpha}
\quad(\alpha\ne0),\qquad
Q_0(\mu)=\exp\left(\frac{\log P+\log Q}{2}\right),
\]
with $Q_0$ the limit as $\alpha\to0$~\cite[Theorem~5.2]{kim-2018}.
They coincide with $A_\alpha(P,Q)$ when $P,Q$ commute or $\alpha=\pm1$, but differ in general.
Their opposite-parameter pair need not preserve $P\#Q$ either.
For example, for
\[
P=\begin{pmatrix}1&0\\0&4\end{pmatrix},\qquad
Q=\begin{pmatrix}5&4\\4&5\end{pmatrix},
\]
direct calculation gives
$Q_{1/2}(\mu)\#Q_{-1/2}(\mu)\ne P\#Q$.

Hiai and Petz~\cite[Theorem~2.1]{hiai-petz-2009} realize $Q_\alpha(\mu)$ as the geodesic midpoint of the Euclidean metric pulled back by $F_\alpha(D)=D^\alpha/\alpha$ for $\alpha\ne0$, and by $F_0(D)=\log D$ at $\alpha=0$.
Their parameter is $\theta=2-2\alpha$; $\alpha=0$ gives the log-Euclidean metric.
Here, by contrast, the affine-invariant metric $g$ is fixed and the connection varies with $\alpha$, with $A_0(P,Q)=P\#Q$.
\end{Rem}

\section{A Hessian counterexample on SPD(2)}\label{sec:counterexample}

Write a matrix in $\SPD(2)$ as
\[ X = \begin{pmatrix}x&z\\z&y\end{pmatrix}, \qquad x > 0, \quad y > 0, \quad xy-z^2 > 0. \]
Consider the Hessian potential and its associated metric
\[ \Phi(x,y,z)  =  \frac{x^4}{12} + \frac{y^4}{12} + \frac{z^2}{2}, \qquad g = \operatorname{Hess}\Phi  =  x^2\,dx^2 + y^2\,dy^2 + dz^2. \]
Let $\nabla^E$ be the flat connection in the coordinates $(x,y,z)$.
Since $g$ is Hessian with respect to $\nabla^E$, its dual connection has affine coordinates
\[ \eta  =  D\Phi  =  \left(\frac{x^3}{3},\frac{y^3}{3},z\right). \]
The potential $\Phi$ is smooth and strictly convex on all of $\mathbb{R}^{3}$, and its gradient $D\Phi$ is a bijection of $\mathbb{R}^{3}$.
Thus $\Phi$ is of Legendre type in the sense of Rockafellar~\cite{rockafellar1967conjugates}.
The Legendre map $\eta=D\Phi$ is a homeomorphism of $\mathbb{R}^{3}$ onto itself, with inverse $D\Phi^{*}$, where $\Phi^{*}$ is the Legendre conjugate of $\Phi$.
Its derivative is $\operatorname{diag}(x^2,y^2,1)$, so its restriction to $\SPD(2)$ is a smooth diffeomorphism onto its image.
It maps $\SPD(2)$ onto
\[ \eta\left(\SPD(2)\right) = \left\{(u,v,w) \in \mathbb{R}^3 : u>0,\ v>0,\ w^6<9uv \right\}. \]
The restriction of $\Phi$ to $\SPD(2)$ is not of Legendre type, since $D\Phi$ stays bounded at the boundary points of $\SPD(2)$.
Thus the convexity of this image does not follow from the general theory of Legendre-type functions~\cite{rockafellar1967conjugates}.
We verify it directly.
The image is the strict epigraph $v>f(u,w)$ of
$f(u,w) = w^6/(9u)$ on $u>0$.
The diagonal entries of $\operatorname{Hess} f$ are nonnegative, and
\[ \det\left(\operatorname{Hess} f\right)=\frac{8w^{10}}{27u^4} \ge 0. \]
Thus $f$ is convex, and hence $\eta\left(\SPD(2)\right)$ is convex.
Consequently,
\[ (1-t)\eta(P)+t\eta(Q) \in \eta\left(\SPD(2)\right), \qquad P,Q\in\SPD(2),\quad 0 \le t \le 1. \]
Accordingly, the dual midpoint mean is globally well-defined and is
\[ M(P,Q) =  \eta^{-1}\left(\frac{\eta(P) + \eta(Q)}{2}\right). \]

Here $D\Phi$ denotes the coordinate gradient, or Legendre map, in the $\nabla^{E}$-affine coordinates.
It should not be confused with the affine connection $\nabla^{E}$.
Thus $M$ is the equal-weight vector quasi-arithmetic mean generated by $D\Phi$.  
Equivalently, it is the midpoint of the $\nabla^{*}$-geodesic, since
$\eta = D\Phi$ is a $\nabla^{*}$-affine coordinate system.

This representation has the usual affine gauge freedom of a Hessian structure.
For the affine and scaling degrees of freedom of dual Hessian--Legendre representations, see \cite[Eq.~(3) and Section 3]{nielsen-2026-legendre}. 

If
\[ \widetilde{\Phi}(\xi)  =  \Phi(\xi) + \langle a,\xi\rangle  +  b,  \qquad \xi = (x,y,z), \]
then
\[ \mathrm{Hess}\,\widetilde{\Phi}  =  \mathrm{Hess}\,\Phi = g, \qquad D\widetilde{\Phi} = D\Phi + a, \]
and therefore
\[ \left(D\widetilde{\Phi}\right)^{-1}\left(\frac{D\widetilde{\Phi}(P) + D\widetilde{\Phi}(Q)}{2}\right)  =  \left(D\Phi\right)^{-1}\left(\frac{D\Phi(P) + D\Phi(Q)}{2}\right)  =  M(P,Q). \]
More generally, replacing the dual affine coordinates by
$\widetilde{\eta}  =  B\eta  +  c$, with $B \in\mathrm{GL}(3,\mathbb{R})$, does not change the point represented by their arithmetic midpoint:
\[ \widetilde{\eta}^{-1} \left(\frac{\widetilde{\eta}(P) + \widetilde{\eta}(Q)}{2}\right)  =  \eta^{-1}\left(\frac{\eta(P) + \eta(Q)}{2}\right).\]
Thus this gauge freedom changes the affine representation and leaves the geometrically determined dual midpoint mean unchanged.
Multiplying $\Phi$ by a positive constant also leaves the quasi-arithmetic mean unchanged.
However, it rescales the metric by that constant, so it is not a gauge transformation when $g$ is fixed.
In the fixed-$g$ setting considered here, the gauge freedoms are addition of an affine function to the potential and compatible affine changes of the primal and dual affine coordinates.
Multiplication of the potential by a positive constant is a homothetic rescaling of the metric.
The parameter $\alpha$ is not a gauge parameter, since changing $\alpha$ changes the affine connection and its midpoint mean.

In particular, the factors $1/3$ in the first two components of $\eta$ do not affect the resulting midpoint.

For $s,t>0$, restricting to the ray of scalar matrices gives
\[ A(sI_2,tI_2) = \frac{s+t}{2}I_2, \qquad M(sI_2,tI_2) = \left(\frac{s^3+t^3}{2}\right)^{1/3}I_2. \]

In order to identify the ambient Levi--Civita midpoint, introduce the coordinates
\[ u=\frac{x^2}{2},\quad v=\frac{y^2}{2},\quad w=z.\]
In these coordinates, the metric becomes Euclidean:
$g = du^2+dv^2+dw^2$. 
Moreover, the scalar ray is represented by
$rI_2 \mapsto \left(\frac{r^2}{2},\frac{r^2}{2},0\right)$. 
The Euclidean segment joining the images of $sI_2$ and $tI_2$
remains in the coordinate image and lies entirely in the scalar ray.
Consequently, this segment is also the ambient Levi--Civita
geodesic. Its midpoint has coordinates $\left(  \frac{s^2+t^2}{4},  \frac{s^2+t^2}{4},  0\right)$,
and therefore
$M_{g} (sI_2,tI_2) = \sqrt{\frac{s^2+t^2}{2}}\,I_2$. 
Thus, on the scalar ray, $A$, $M$, and $M_g$ are the power means of orders $1$, $3$, and $2$, respectively.

For the initial matrices $P_0 = I_2,\ Q_0 = 2I_2$,
write
$ P_n = a_n I_2, \quad Q_n = b_n I_2$.
The arithmetic--dual-mean iteration then reduces to
\[ a_{n+1} = \frac{a_n+b_n}{2},\quad b_{n+1}  = \left(\frac{a_n^3+b_n^3}{2}\right)^{1/3},\quad (a_0,b_0)=(1,2).\]

For $0 < a < b$, the strict power-mean inequality yields $a < \frac{a+b}{2} < \left(\frac{a^3+b^3}{2}\right)^{1/3} < b$. 
Consequently, $a_n < a_{n+1} < b_{n+1} < b_n$
whenever $a_n<b_n$. Thus, $(a_n)$ is increasing and
$(b_n)$ is decreasing, and there exist numbers
$\alpha \le \beta$ such that
$a_n \to \alpha, \ b_n \to \beta, n \to \infty$. 
Passing to the limit in the first recurrence relation yields $\alpha = \frac{\alpha+\beta}{2}$, and hence $\alpha=\beta$. 
Denoting their common value by $L$,
we obtain
\[  (P_n,Q_n) \to (LI_2,LI_2), \ n \to \infty, \quad (A\otimes M)(I_2,2I_2) = LI_2.\]

It remains to compare $L$ with the Levi--Civita midpoint.
The first two iterates satisfy $a_1 = \frac{3}{2}, \  b_1^3 = \frac{9}{2}$,
and therefore
\[ b_2^3 = \frac{a_1^3+b_1^3}{2} = \frac12\left(\frac{27}{8}+\frac{9}{2}\right) = \frac{63}{16}.\]
Since
\[b_2^6 = \frac{3969}{256} <  \frac{4000}{256} =  \left(\sqrt{\frac52}\right)^6,\]
we obtain $L \leq b_2<\sqrt{\frac{5}{2}}$.

On the other hand, $M_g(I_2,2I_2) = \sqrt{\frac{5}{2}}\,I_2$.
Therefore,
\[ (A\otimes M)(I_2,2I_2) = LI_2 \ne \sqrt{\frac{5}{2}}\,I_2 = M_g(I_2,2I_2). \]
Numerically,
\[ L\approx1.5772868875, \quad \sqrt{\frac{5}{2}}\approx1.5811388301.\]

This example satisfies the dualistic relation $\nabla^*  =  2\nabla^g - \nabla^E$.
Thus $A$ and $M$ have not been chosen independently.
The example shows that duality alone does not imply that the Gauss composition equals the Levi--Civita midpoint.
Moreover, one-step midpoint invariance fails for every pair of distinct scalar matrices, including arbitrarily close pairs.
Write $s=r-a$, $t=r+a$, where $0<a<r$, and use $A$, $M$, and $M_g$ also for their scalar restrictions.
Then
\[ A(s,t)=r, \qquad M(s,t)=r\left(1+\frac{3a^2}{r^2}\right)^{1/3}. \]
The strict concavity of $u\mapsto u^{2/3}$ on $(0,\infty)$ gives
\[
\begin{aligned}
M_g\left(A(s,t),M(s,t)\right)^2
&=\frac{r^2}{2}\left[1+\left(1+\frac{3a^2}{r^2}\right)^{2/3}\right]\\
&<r^2+a^2=M_g(s,t)^2.
\end{aligned}
\]
Thus dual flatness does not imply local midpoint invariance.
For any such pair, there is no isometry and invariant domain satisfying all the hypotheses of Theorem~\ref{thm:symmetry}.
Indeed, the theorem would imply the one-step invariance, which fails for this pair.

\section{A Euclidean metric with a parallel cubic form}\label{sec:parallel-cubic}

In this section, we give a construction valid in every dimension $d \ge 2$.
It includes both flat and nonflat examples.
The nonflat examples show that the point-symmetry criterion also applies to structures which are not dually flat.
Let $\mathcal{P} = \mathbb{R}^{d}$ carry the Euclidean metric $g = \langle\cdot,\cdot\rangle$, and let $C$ be a nonzero totally symmetric trilinear form with constant coefficients, or equivalently, a $\nabla^{g}$-parallel cubic form.
Define the constant symmetric bilinear map $K$ by
\[  \langle K(X,Y),Z\rangle = C(X,Y,Z),\]
and set
\[  \nabla \coloneqq \nabla^{E} + K,
  \qquad
  \nabla^{*} \coloneqq \nabla^{E} - K,\]
where $\nabla^{E}=\nabla^g$ is the flat Euclidean connection.
The total symmetry of $C$ makes both connections torsion-free and shows that $\nabla^{*}$ is the $g$-dual of $\nabla$.
More explicitly,
\[  (\nabla_X g)(Y,Z) = -2C(X,Y,Z),\]
which is totally symmetric.
Thus $(\mathbb{R}^{d}, g, \nabla)$ is a statistical manifold in the sense of Lauritzen~\cite{lauritzen-1987}.
The average connection of the dual pair is
\[ \overline{\nabla}  = \frac{1}{2} (\nabla+\nabla^{*}) = \nabla^{E} = \nabla^{g},\]
in accordance with Section~\ref{sec:setting}.
Consequently,  $M_g (P,Q)=\frac{P+Q}{2}$, 
while $K = \nabla-\nabla^g$ is exactly the tensor of Remark~\ref{rem:symmetry}(3).

For constant vector fields, writing $K_X\coloneqq K(X,\cdot)$, 
we have
\[  R^{\nabla}(X,Y)Z = [K_X,K_Y]Z = R^{\nabla^{*}}(X,Y)Z. \]
Thus $\nabla$ and $\nabla^{*}$ are flat if and only if
\[  [K_X,K_Y] = 0 \quad  \text{for all } X,Y\in\mathbb{R}^{d}.\]
On $\mathbb{R}^{2}$, define $C$, up to total symmetry, by
\[
  \begin{aligned}
  C(\partial_x,\partial_x,\partial_y)&=1,\\
  C(\partial_x,\partial_x,\partial_x)
  &=C(\partial_x,\partial_y,\partial_y)
   =C(\partial_y,\partial_y,\partial_y) = 0.
  \end{aligned}
\]
Then
\[ K(\partial_x,\partial_x)=\partial_y,  \qquad  K(\partial_x,\partial_y)=\partial_x,  \qquad K(\partial_y,\partial_y)=0.\]
The geodesic equations are
\[  \ddot{x}=-2\dot{x}\dot{y}, \qquad  \ddot{y}=-\dot{x}^{2}  \qquad \text{for }\nabla,\]
and
\[ \ddot{x} = 2\dot{x}\dot{y}, \qquad  \ddot{y} = \dot{x}^{2} \qquad \text{for }\nabla^{*}.\]
Moreover,
\[  R^{\nabla}(\partial_x,\partial_y)\partial_y  = -K(\partial_y,K(\partial_x,\partial_y)) = -\partial_x \ne 0.\]
For $d > 2$, extend this cubic form by zero whenever one of its arguments belongs to the orthogonal complement of $\operatorname{span}\{\partial_x,\partial_y\}$.
The same nonzero curvature component remains, giving a nonflat example in every dimension $d \ge 2$.
For such a nonflat choice, neither $\nabla$ nor $\nabla^{*}$ admits local affine coordinates in which its connection coefficients vanish on an open set.
Thus the affine-coordinate quasi-arithmetic representations of Section~\ref{sec:setting} do not exist for $A$ and $M$.

Fix $P \ne Q$, set $o\coloneqq M_g(P,Q)=\frac{P+Q}{2}$,
and let $s(x) \coloneqq 2o-x$  be the Euclidean point reflection at $o$.
Then $s$ is a global isometry of $g$ with
\[  s(o)=o,\qquad ds\equiv-\operatorname{id}, \qquad  \operatorname{Fix}(s) = \{o\}.\]
Thus conditions (a) and (c) of Theorem~\ref{thm:symmetry} hold on every $s$-invariant domain containing $o$.
Moreover, $s$ preserves $\nabla^E$ and, since $ds=-\operatorname{id}$, we have $s^*C=-C$ and, equivalently, $s_*K=-K$.
Therefore, $s_*\nabla=\nabla^E-K=\nabla^{*}$, which is condition (b).

The same argument verifies condition (b) on a connected Riemannian symmetric space with a $\nabla^g$-parallel cubic form $C$.
Indeed, if $s_o$ denotes the geodesic symmetry at $o$, then $\left.(s_o^*C)\right|_o = -C_o$  because $d_o s_o = -\operatorname{id}$.
Since $s_o$ is an isometry, both $s_o^*C$ and $-C$ are $\nabla^g$-parallel, and hence they agree on the connected component of $o$.
Thus $s_{o*}K = -K$ and $s_{o*}\nabla=\nabla^*$.
The resulting invariance statement is local for sufficiently close pairs unless the global domain and fixed-point assumptions of Theorem~\ref{thm:symmetry} are verified separately.
The same change of sign is used in Section~\ref{sec:nakamura}.

We next construct a domain on which the required local midpoint branches are uniquely defined.
Let
\[  \kappa\coloneqq\max_{\|v\|=1}\|K(v,v)\| > 0.\]
Along every nonconstant $\nabla$- or $\nabla^{*}$-geodesic,
\[  \frac{d^2}{dt^2}\|\gamma-o\|^2 = 2\|\dot\gamma\|^2  \mp2\langle K(\dot\gamma,\dot\gamma),\gamma-o\rangle   \ge  2\|\dot\gamma\|^2\left(1-\kappa\|\gamma-o\|\right),\]
where the upper and lower signs correspond to $\nabla$ and $\nabla^*$, respectively.
Hence $\|\gamma-o\|^2$ is strictly convex on every interval on which $\|\gamma-o\|<\kappa^{-1}$.

For an affine connection $D$, we call an open set $V$ a \emph{$D$-convex normal neighborhood} if any two points of $V$ are joined by a unique $D$-geodesic whose segment is contained in $V$, and $V$ is a normal neighborhood of each of its points with respect to $D$.
Unlike the strongly convex normal neighborhoods of Section~\ref{sec:GC}, no minimality is required here: an affine connection alone does not determine a length functional.
Every point has arbitrarily small $D$-convex normal neighborhoods~\cite{whitehead-1932-convex,kobayashi-nomizu-1963}.
By this existence and translation invariance, there exists $\rho > 0$, depending only on $C$, such that for every $o\in\mathbb R^d$ and each $D \in \{\nabla,\nabla^*,\nabla^E\}$, 
there is a $D$-convex normal neighborhood $V_o^D$ satisfying
$B_\rho(o) \subset V_o^D \subset B_{\kappa^{-1}}(o)$. 
Put $\delta\coloneqq\|P-Q\|$ and assume that $0<\delta<\rho$.
The endpoints $P$ and $Q$ belong to every $V_o^D$.
For $D = \nabla$ or $\nabla^*$, strict convexity along the unique $D$-geodesic in $V_o^D$ joining $P$ and $Q$ gives
$\|\gamma(t)-o\|\leq\frac{\delta}{2}$ for $0 \le  t \le 1$. 
Thus this segment is contained in the $s$-invariant ball $U_{P,Q}\coloneqq B_\delta(o)$, 
and it is the unique $D$-geodesic joining $P$ and $Q$ that remains in $U_{P,Q}$.
The Euclidean segments defining the two required Riemannian midpoints also remain in $U_{P,Q}$.

The preceding argument establishes uniqueness for geodesic segments contained in $U_{P,Q}$.
It does not rule out ambient geodesics that may leave $U_{P,Q}$.
Accordingly, we apply Theorem~\ref{thm:symmetry} to the metric and connections restricted to $U_{P,Q}$.
Denoting the resulting local midpoint branches again by $A$, $M$, and $M_g$, we obtain
\begin{equation}\label{eq:sum-invariance}
  M_g\left(A(P,Q),M(P,Q)\right)  = M_g(P,Q), \ \text{i.e.,}\   A(P,Q)+M(P,Q) = P+Q.
\end{equation}
This identity follows because the point reflection $s$ exchanges the two local dual midpoints.
In real dimension two, \eqref{eq:sum-invariance} has the same form as the holomorphic Matkowski--Sut\^{o} equation on $\mathbb C$.
Here the midpoint maps are constrained by metric duality, whereas holomorphic quasi-arithmetic means arise from local flat affine coordinates and need not satisfy this duality condition.
Thus our non-flat examples give a different geometric realization of the same invariance equation.
Identity~\eqref{eq:sum-invariance} is a multidimensional invariance equation of Matkowski--Sut\^{o} type.
Here it is satisfied by midpoint maps which need not have an affine-coordinate quasi-arithmetic representation. 
The scalar equation~\eqref{eq:MS} and its complete solution are recalled in Section~\ref{sec:one-dim}.

For every fixed $\alpha\in\mathbb R$, let
\[  \nabla^{(\alpha)}=\nabla^E+\alpha K,\quad  A_\alpha\coloneqq M_{\nabla^{(\alpha)}}, \quad  M_\alpha\coloneqq M_{\nabla^{(-\alpha)}}.\]
Remark~\ref{rem:symmetry}(3) and the same reflection give $A_\alpha(P,Q)+M_\alpha(P,Q) = P+Q$
whenever the corresponding local midpoint branches are defined on a suitable invariant neighborhood.
The admissible neighborhood may depend on $\alpha$, and a common neighborhood can be chosen when $\alpha$ ranges over a fixed bounded set.

The Gauss composition is also well defined for sufficiently close pairs.
Put $h\coloneqq Q-P$.
In order to obtain the local expansions, we write the relevant geodesic as
\[ \gamma(t) = P + th + u(t), \quad u(0) = u(1) = 0.\]
The geodesic equation reads $\ddot{\gamma} = \mp K(\dot{\gamma},\dot{\gamma})$, where the upper and lower signs again correspond to $\nabla$ and $\nabla^{*}$.

We first show the a priori estimate
\[ \|u\|_{C^{1}([0,1])} = O\left(\|h\|^{2}\right). \]
Indeed, $\dot{\gamma}(0) = \left(\operatorname{Exp}^{D}_{P}\right)^{-1}(Q)$ for the corresponding connection $D \in \{\nabla, \nabla^{*}\}$ depends smoothly on $Q$ and vanishes at $Q = P$, so $\|\dot{\gamma}(0)\| = O(\|h\|)$. 
By the differential inequality $\tfrac{d}{dt}\|\dot{\gamma}\| \le \|K(\dot{\gamma},\dot{\gamma})\| \le 2\kappa \|\dot{\gamma}\|^{2}$, the bound $\|\dot{\gamma}(t)\| = O(\|h\|)$ holds on $[0,1]$ for sufficiently small $\|h\|$.
Hence $\|\ddot{u}(t)\| = \|K(\dot{\gamma},\dot{\gamma})(t)\| = O\left(\|h\|^{2} \right)$.
Since $u(0) = u(1) = 0$, we have $\int_{0}^{1}\dot{u}(s)\,ds = 0$.
Writing $\dot{u}(t) = \int_{0}^{1}\left(\dot{u}(t) - \dot{u}(s)\right)ds$, we obtain $\|\dot{u}\|_{\infty} = O(\|h\|^{2})$.
Then $u(t) = \int_{0}^{t}\dot{u}(s)\,ds$ gives $\|u\|_{\infty} = O\left(\|h\|^{2} \right)$.
Substituting $\dot{\gamma} = h + \dot{u}$ into the geodesic equation, 
the cross terms satisfy $2K(h,\dot{u}) = O\left(\|h\|^{3}\right)$ and $K(\dot{u},\dot{u}) = O\left(\|h\|^{4}\right)$, so
\[  \ddot{u}(t) = \mp K(h,h) + O\left(\|h\|^3\right).\]

Solving the leading boundary-value problem $\ddot{u}_{0}(t) = \mp K(h,h)$ with $u_{0}(0) = u_{0}(1) = 0$ by two integrations gives
\[ u_{0}(t) = \pm\frac{1}{2}\,t(1-t)\,K(h,h), \qquad u_{0}\left(\frac{1}{2}\right) = \pm\frac{1}{8}\,K(h,h). \]
The $O\left(\|h\|^{3}\right)$ remainder in $\ddot{u}$ contributes $O\left(\|h\|^{3}\right)$ to $u$.
Indeed, $v \coloneqq u - u_{0}$ satisfies $\ddot{v} = O\left(\|h\|^{3}\right)$ with $v(0) = v(1) = 0$.
The Green function of this two-point Dirichlet problem on $[0,1]$ is bounded; see, e.g., \cite[Chapter~7]{walter-1998-ode}.
Hence
\[ u\left(\frac{1}{2}\right) = \pm\frac{1}{8} K(h,h) + O\left(\|h\|^3 \right).\]

Since the local midpoint branches are smooth and \eqref{eq:sum-invariance} holds for every sufficiently close pair, Theorem~\ref{thm:local-existence} gives local quadratic convergence of the mean-type iteration.
By Theorem~\ref{thm:local-convergence}, we can identify the limit.
There exists $\varepsilon>0$ such that
\[ (A\otimes M)(P,Q) = M_g(P,Q) =\frac{P+Q}{2}, \qquad (P,Q) \in \Omega_{\varepsilon}, \]
where
\[ \Omega_{\varepsilon} \coloneqq  \left\{(P,Q)\in\mathbb{R}^d \times \mathbb{R}^d \colon \|P-Q\|<\varepsilon\right\}. \]
Here $\varepsilon$ and the constants in the convergence estimates do not depend on the center $o$, because the structure is invariant under translations.
The computation above gives the leading quadratic term explicitly:
\[  A(P,Q) = o + \frac{1}{8} K(h,h)+ O\left(\|h\|^3 \right), \qquad M(P,Q)  = o - \frac{1}{8} K(h,h)+O\left(\|h\|^3 \right),\]
so that $\|A(P,Q)-M(P,Q)\|\le\left(\frac{\kappa}{4}+O(\|h\|)\right)\|h\|^{2}$ with $\kappa$ as above.
It also shows directly that, by the invariance \eqref{eq:sum-invariance}, the arithmetic mean $\frac{P_n+Q_n}{2}=o$ is fixed along the iteration.
This is how Theorem~\ref{thm:local-convergence} applies in this example.

The relation to the other examples is as follows.
If $\nabla$ is flat, that is, if the operators $K_X$ commute, the structure splits into an orthogonal product of one-dimensional structures and reduces, up to gauge, to products of the scalar structures classified in Section~\ref{sec:one-dim}. 
See Remark~\ref{rem:examples-revisited}(5) for the precise statement.
However, when the operators $K_X$ do not commute, the structure is neither dually flat nor a product of one-dimensional structures.
In this case,~\eqref{eq:sum-invariance} does not follow from the one-dimensional classification of Section~\ref{sec:one-dim}.
It follows from point symmetry alone.
Together with Section~\ref{sec:counterexample}, this shows that dual flatness is neither sufficient nor necessary for the local invariance identity.
The hypothesis used here is the existence of the symmetry in Theorem~\ref{thm:symmetry}.
Conversely, one may ask whether every Euclidean dual pair with the local invariance identity arises from a constant cubic form in this way.
We state this as Conjecture~\ref{conj:rigidity} in Section~\ref{sec:conjecture}.

\section{The one-dimensional case}\label{sec:one-dim}

In this section, we give a complete characterization of the invariance identity \eqref{eq:Mg-inv} of Theorem~\ref{thm:invariance} in dimension one.
We first give a normal form for one-dimensional dualistic structures which represents the three midpoint means as quasi-arithmetic means.
We then apply the solution of the {\it Matkowski--Sut\^{o} problem} in the theory of functional equations.
All statements in this section are known or follow from known results by short arguments.
We include them to explain how the scalar computations of Sections~\ref{sec:nakamura} and~\ref{sec:counterexample} fit into the complete one-dimensional classification.
Together with Theorem~\ref{thm:restriction} below, this classification also gives necessary conditions in arbitrary dimension.
These complement the sufficient criterion of Section~\ref{sec:symmetry}.

Recall that the quasi-arithmetic mean generated by a continuous strictly monotone real function $f$ on an interval is
\[ M_f(p,q) \coloneqq f^{-1}\left(\frac{f(p) + f(q)}{2}\right), \]
and that $M_{f_1} = M_{f_2}$ if and only if $f_2 = \alpha f_1 + \beta$ with $\alpha \ne  0$~\cite{hardy-littlewood-polya-1952}. 
By Section~\ref{sec:setting}, the midpoint mean of a flat torsion-free connection is quasi-arithmetic in any affine coordinate.

The Riemannian part of the normal form below is closely related to Nielsen's description of one-dimensional Karcher quasi-arithmetic means~\cite[Propositions~1 and~2]{nielsen-2025-frechet}.
Indeed, for a one-dimensional metric $g = w(t)\,dt^2$, the arc-length coordinate
$s(t) = \int^{t} \sqrt{w(\tau)}\,d\tau$ gives $d_g(p,q) = |s(p)-s(q)|$.
Consequently, the Karcher midpoint is the quasi-arithmetic mean
\[ M_g(p,q) = M_s(p,q) =  s^{-1}\left(\frac{s(p)+s(q)}{2}\right).\]
For the Hessian realization $w = \Phi^{\prime\prime}$, 
Nielsen also expresses this same Karcher midpoint in Legendre-dual coordinates. 
These dual expressions represent the same Riemannian midpoint.
They should be distinguished from the two generally distinct affine midpoint maps $M_\theta$ and $M_\eta$ associated here with $\nabla$ and $\nabla^*$, respectively.
In the following lemma, we describe all three affine coordinates $\theta$, $\eta$, and $s$, and give the duality relation $\theta'\eta'=(s')^2$.

Throughout this section, $I \subset \mathbb{R}$ is an open interval equipped with a Riemannian metric $g = w(t)\,dt^2$, 
where $w > 0$ is smooth, and an affine connection $\nabla$ with Christoffel symbol $\Gamma$, that is, $\nabla_{\partial_t}\partial_t = \Gamma(t)\,\partial_t$. 
In dimension one, every affine connection is torsion-free and flat, because the torsion and curvature tensors are antisymmetric in two of their arguments. 
Moreover, as shown below, any two points of $I$ are joined by a unique $\nabla$-geodesic up to affine reparametrization, so the standing assumptions of Section~\ref{sec:setting} hold automatically.

Here $\int^{t} \Gamma(\tau)\,d\tau$ denotes any fixed primitive function of $\Gamma$ on $I$; different choices change $\theta, \eta, s$ only by affine transformations.

\begin{Lem}[One-dimensional normal form]\label{lem:normal-form}
Let $\theta, \eta, s \colon I \to \mathbb{R}$ be primitive functions of
\[ \theta'(t) = \exp\left(\int^{t}\Gamma(\tau)\,d\tau\right), \qquad \eta'(t) = \frac{w(t)}{\theta'(t)}, \qquad s'(t) = \sqrt{w(t)}, \qquad t \in I. \]
Choose primitive functions satisfying the displayed normalizations. 
Their additive constants are arbitrary, and more general affine changes of the generators do not alter the associated quasi-arithmetic means. 
Then, \\
(i) The $g$-dual and Levi--Civita connections have Christoffel symbols
\[ \Gamma^{*}(t) = \frac{w'(t)}{w(t)} - \Gamma(t), \qquad \Gamma^{g}(t) = \frac{w'(t)}{2w(t)}. \]
(ii) $\theta$, $\eta$, and $s$ are affine coordinates for $\nabla$, $\nabla^{*}$, and $\nabla^{g}$, respectively, and
\[ A = M_{\theta}, \qquad M = M_{\eta}, \qquad M_g = M_{s} \qquad \text{on } I \times I. \]
(iii) With the above normalization, $\theta'\eta'  =  w  =  (s')^{2}$.\\
(iv) In the $\nabla$-affine coordinate $x = \theta(t)$, the metric takes the Hessian form $g = \Phi''(x)\,dx^{2}$ with $\eta = \Phi'\circ\theta$, for a strictly convex potential $\Phi$ determined up to the affine gauge of Section~\ref{sec:counterexample}. 
Thus every one-dimensional dualistic structure is dually flat, i.e.\ Hessian~\cite{AmariNagaoka2000,shima-2007}.
\end{Lem}

\begin{proof}
A curve $u \mapsto t(u)$ in $I$ is a $\nabla$-geodesic if and only if $\ddot{t}(u) + \Gamma(t(u))\,\dot{t}(u)^{2} = 0$, so for smooth $f$, along any $\nabla$-geodesic,
\[ \frac{d^{2}}{du^{2}}\,f(t(u)) = f''(t(u))\,\dot{t}(u)^{2} + f'(t(u))\,\ddot{t}(u) = \left(f''-\Gamma f'\right)(t(u))\,\dot{t}(u)^{2}. \]
Hence $f$ is an affine coordinate for the connection with symbol $\Gamma$ if and only if $f'' = \Gamma f'$ on $I$.
This is equivalent to $f'$ being a nonzero multiple of $\exp\left(\int^{t}\Gamma(\tau)\,d\tau\right)$.
We call this equivalence the \emph{affine-coordinate criterion}.
Affine coordinates are unique up to affine transformations. 
Evaluating the duality identity \eqref{eq:duality} on the coordinate frame ($X = Y = Z = \partial_t$) gives $w'(t) = \left(\Gamma(t) + \Gamma^{*}(t)\right)w(t)$, which proves the first formula in (i). 
The second formula is the usual one-dimensional formula for the Christoffel symbol of the Levi--Civita connection.
Together, these formulas give $\nabla^{*} = 2\nabla^{g}-\nabla$.

Applying the affine-coordinate criterion to $\Gamma$, $\Gamma^{*}$, and $\Gamma^{g}$, with $\Gamma^{*}$ and $\Gamma^{g}$ given by (i), yields
\[ \exp\left(\int^{t}\Gamma^{*}(\tau)\,d\tau\right) =  c_1\, w(t)\exp\left(-\int^{t}\Gamma(\tau)\,d\tau\right) = c_2\,\frac{w(t)}{\theta'(t)}, \]
\[  \exp\left(\int^{t}\Gamma^{g}(\tau)\,d\tau\right) = c_3\sqrt{w(t)}, \]
with constants $c_i>0$, which proves (ii). 
A $\nabla$-geodesic is an affinely parametrized segment in the coordinate $\theta$, so any $p,q\in I$ are joined by a unique such geodesic inside $I$, whose midpoint is $M_{\theta}(p,q)$, and similarly for $\eta$ and $s$. 

Statement (iii) is immediate from the defining formulas. 

For (iv), in the coordinate $x = \theta(t)$ the metric coefficient $\widetilde{w}$ is determined by $\widetilde{w}(\theta(t)) = w(t)/\theta'(t)^{2}$, and, regarding $\eta$ as a function of $x$,
\[ \frac{d\eta}{dx}\left(\theta(t)\right) = \frac{\eta'(t)}{\theta'(t)} = \frac{w(t)}{\theta'(t)^{2}} = \widetilde{w}\left(\theta(t)\right)>0, \]
so any $\Phi$ with $\Phi'' = \widetilde{w}$ satisfies $\eta = \Phi'\circ\theta$ up to an additive constant.
\end{proof}

Statement (iv) is special to dimension one in the following sense.
By a theorem of Bryant, every smooth surface metric is locally Hessian with respect to some flat connection~\cite{bryant2025hessianizability}.
In dimension one, this flat connection can be taken to be the prescribed connection $\nabla$, as stated in (iv).
In higher dimensions, the connection $\nabla$ of a dualistic structure is prescribed.
The pair $(g,\nabla)$ need not be dually flat even when $g$ is Hessian with respect to another flat connection.
The Euclidean examples of Section~\ref{sec:parallel-cubic} with noncommuting $K_X$ are examples of this situation.

\begin{Prop}[Reduction to a functional equation]\label{prop:MS-reduction}
Let $J\coloneqq s(I)$ and define the generators
\[ \varphi\coloneqq\theta\circ s^{-1}, \qquad \psi\coloneqq\eta\circ s^{-1} \qquad \text{on } J. \]
Then the invariance identity
\[ M_g\left(A(P,Q),M(P,Q)\right)  =  M_g(P,Q), \qquad P,Q\in I, \]
holds if and only if
\begin{equation}\label{eq:MS}
M_{\varphi}(u,v) + M_{\psi}(u,v)  =  u + v, \qquad u,v\in J.
\end{equation}
Moreover, the duality relation of Lemma~\ref{lem:normal-form}(iii) translates into the constraint
\begin{equation}\label{eq:MS-constraint}
\varphi' (u) \psi' (u) = 1, \quad  u \in J, 
\end{equation}
for the chosen dualistic normalization. 
Under arbitrary independent affine changes of the two generators, the product is a nonzero constant and need not be positive. 
Under compatible affine changes preserving the common orientation, it is positive, and, 
under the reciprocal affine gauge preserving the above normalization, it remains equal to one.
\end{Prop}

\begin{proof}
For continuous strictly monotone $f$ and $h$, the definition of the quasi-arithmetic mean gives the conjugation rule
\[ f\left(M_{h}(p,q)\right) = M_{h\circ f^{-1}}\left(f(p),f(q)\right). \]
Apply $s$ to both sides of the invariance identity and put $u = s(P)$, $v = s(Q)$. Since $s(M_{s}(a,b)) = \frac{s(a) + s(b)}{2}$, the left-hand side becomes
\[ \frac{1}{2}\left(s\left(A(P,Q)\right) + s\left(M(P,Q)\right)\right) = \frac{1}{2}\left(M_{\varphi}(u,v) + M_{\psi}(u,v)\right) \]
by the conjugation rule applied to $h = \theta$ and $h = \eta$, while the right-hand side becomes $\frac{u + v}{2}$. Since $s$ is injective, the two identities are equivalent. 
Finally, at $t = s^{-1}(u)$, $\varphi'(u) = \frac{\theta'(t)}{s'(t)}$ and $\psi'(u) = \frac{\eta'(t)}{s'(t)}$. 
Therefore $\varphi'(u)\,\psi'(u) = \frac{\theta'(t)\,\eta'(t)}{s'(t)^{2}} = 1$ by Lemma~\ref{lem:normal-form}(iii).
\end{proof}

Equation~\eqref{eq:MS} means that the arithmetic mean is invariant with respect to the mean-type mapping $(M_{\varphi},M_{\psi})$.
In other words, one application of the pair does not change the arithmetic mean.
The problem of determining all pairs of quasi-arithmetic means with this property is called the \emph{Matkowski--Sut\^{o} problem}.
Sut\^{o}~\cite{suto-1914-1,suto-1914-2} solved it under analyticity assumptions, and Matkowski~\cite{matkowski-1999} solved it under additional regularity assumptions.
Dar\'{o}czy and P\'{a}les~\cite{daroczy-pales-2002} gave a complete solution assuming only continuity and strict monotonicity. 
See the survey~\cite{jarczyk-jarczyk-2018} for the invariance problem in general.

\begin{Rem}[Sut\^{o}'s invariance problem]\label{rem:suto}
Sut\^{o}~\cite[\S 3]{suto-1914-2} considers the invariance identity
\[
M_\chi\left(M_{\chi_1}(p,q),M_{\chi_2}(p,q)\right)=M_\chi(p,q)
\]
and reduces it to \eqref{eq:MS} by conjugation with $\chi$.
This is the change of variables in Proposition~\ref{prop:MS-reduction}, with $\chi=s$, $\chi_1=\theta$, and $\chi_2=\eta$.
His three generators are independent, whereas ours satisfy the duality constraint \eqref{eq:MS-constraint}.
Under analyticity assumptions, his classification gives the exponential and affine pairs of Theorem~\ref{thm:DP} below.

Taking $\chi=\log$, Sut\^{o} obtains invariance of the geometric mean under the power means of orders $\alpha$ and $-\alpha$, including the arithmetic--harmonic pair~\cite[\S 3]{suto-1914-2}.
This is the scalar restriction of the invariance identities in Section~\ref{sec:nakamura}.
The matrix iteration and its convergence are treated by Nakamura~\cite{nakamura-2001} for the arithmetic--harmonic pair.
Theorem~\ref{thm:rigidity} interprets the scalar classification through the metric coefficient $w$ and uses Theorem~\ref{thm:DP} to dispense with analyticity.
\end{Rem}

\begin{Thm}[Dar\'{o}czy--P\'{a}les~{\cite[Theorem 4.12]{daroczy-pales-2002}}]\label{thm:DP}
Let $J\subset\mathbb{R}$ be an open interval and let $\varphi,\psi\colon J\to\mathbb{R}$ be continuous and strictly monotone. Then \eqref{eq:MS} holds if and only if there exists $p\in\mathbb{R}$ such that, up to affine transformations of the generators,
\[ \varphi = \chi_{p}, \qquad \psi = \chi_{-p}, \quad \text{ where } \chi_{p}(u)\coloneqq
\begin{cases}
e^{pu}, & p \ne  0,\\
u, & p = 0.
\end{cases} \]
\end{Thm}

The representatives displayed in Theorem~\ref{thm:DP} do not themselves satisfy the normalization~\eqref{eq:MS-constraint} when $p \ne 0$. 
Indeed, $\chi_p'(u)\chi_{-p}'(u) = -p^2$.
This is consistent with Theorem~\ref{thm:DP}, which determines the generators only up to independent affine transformations.
These transformations do not change the corresponding quasi-arithmetic means.

For $p \ne 0$, consider the affinely equivalent representatives
\[ \widehat\chi_p(u)\coloneqq\frac{1}{p}e^{pu}, \quad \widehat\chi_{-p}(u)\coloneqq-\frac{1}{p}e^{-pu}. \]
They satisfy
\[ \widehat\chi_p'(u)=e^{pu}, \quad \widehat\chi_{-p}'(u)=e^{-pu}, \quad \widehat\chi_p'(u)\widehat\chi_{-p}'(u) \equiv 1. \]
Thus both normalized generators are strictly increasing,
independently of the sign of $p$. For $p=0$, the choice
$\widehat\chi_0(u)=u$ already satisfies the same identity.

More generally, if
\[ \widetilde\varphi = a\chi_p + b, \quad \widetilde\psi = c\chi_{-p} + d, \quad ac \ne 0, \]
then, for $p \ne 0$, $\widetilde\varphi' \widetilde\psi' = - a c p^2$.
Hence the normalization $\widetilde\varphi'\widetilde\psi'\equiv1$ is equivalent to $ac = -p^{-2}$. 
For $p=0$, it is equivalent to $ac=1$.
Therefore, \eqref{eq:MS-constraint} is a normalization of the
generators rather than a property of their unrestricted affine
equivalence classes.

Once this normalization has been imposed, the remaining
compatible affine gauge is
\[ \varphi \longmapsto \lambda\varphi+\beta, \quad \psi\longmapsto\lambda^{-1}\psi+\delta, \quad \lambda \ne 0, \]
which preserves $\varphi'\psi'\equiv1$. Equivalently, once
$\varphi$ is fixed with this normalization, duality determines
$\psi$ up to an additive constant through
$\psi(u)=\psi(u_0) + \int_{u_0}^{u}\frac{dr}{\varphi'(r)}$. 
Thus, for pairs arising from a dualistic structure,
Theorem~\ref{thm:DP} is effectively a condition on the single
normalized generator $\varphi$.

\begin{Thm}[One-dimensional rigidity]\label{thm:rigidity}
Let $(I,g = w\,dx^{2},\nabla)$ be a one-dimensional structure written in a $\nabla$-affine coordinate $x$, so that $A(x_1,x_2) = \frac{x_1 + x_2}{2}$. 
Then the following are equivalent:\\
(a) the invariance identity $M_g \circ (A,M)  =  M_g$ holds on $I\times I$.\\
(b) $A\otimes M  =  M_g$ on $I\times I$, the Gauss composition being well defined because $A$ and $M$ are continuous strict means on an interval.\\
(c) either $w$ is constant on $I$, in which case $\nabla = \nabla^{g}$ and $A = M = M_g$; or 
\begin{equation}\label{eq:inverse-square}
 w(x) = \frac{c}{(x-x_{0})^{2}}
\end{equation}
for some $c > 0$ and $x_{0} \in \mathbb{R}\setminus I$.
In the second case of (c), the affine change $y = \pm(x-x_{0})$ maps $I$ into $(0,\infty)$, and
\[ A(y_1,y_2) = \frac{y_1 + y_2}{2}, \qquad M(y_1,y_2) = \frac{2}{y_1^{-1} + y_2^{-1}}, \qquad M_g(y_1,y_2) = \sqrt{y_1 y_2}: \]
up to affine gauge and a constant rescaling of $g$, the structure is the scalar arithmetic--harmonic--geometric structure underlying Nakamura's iteration (Section~\ref{sec:nakamura}).
\end{Thm}

\begin{proof}
(a)$\Leftrightarrow$(b). The means $A$ and $M$ are quasi-arithmetic, hence continuous and strict. By Matkowski's invariance principle for iterates of mean-type mappings~\cite{matkowski-1999-iterations}, the iteration $P_{n + 1} = A(P_n,Q_n)$, $Q_{n + 1} = M(P_n,Q_n)$ then converges to a common limit for all $P,Q\in I$. Theorem~\ref{thm:invariance} gives the equivalence.

(a)$\Leftrightarrow$(c). By Proposition~\ref{prop:MS-reduction} and Theorem~\ref{thm:DP}, (a) holds if and only if $\varphi = \theta\circ s^{-1}$ is, up to an affine transformation, equal to $\chi_{p}$ for some $p\in\mathbb{R}$.
Indeed, by the normalization $\varphi'\psi'=1$, the generator $\psi$ is an affine transform of $\chi_{-p}$.
This follows by integrating $\psi'=1/\varphi'$.
Thus it suffices to state the condition in terms of $\varphi$ alone.
In the coordinate $x$ we have $\theta = x$ and $\Gamma\equiv 0$.

If $p = 0$, then $x$ is an affine function of $s$.
Thus $s'$ is constant and $w = (s')^{2}$ is constant.
Then $\Gamma^{g} = \frac{w'}{2w}\equiv 0 = \Gamma$, that is, $\nabla = \nabla^{g}$.
It follows that $\nabla^{*} = \nabla^{g}$ and $A = M = M_g$.
Hence (a) holds by idempotence.

If $p \ne  0$, write $x = \alpha e^{ps} + \beta$ with $\alpha \ne  0$ and set $x_{0} = \beta$. 
Then $x-x_{0} = \alpha e^{ps}$ has constant sign on $I$, so $x_{0}\notin I$, and
\[ s(x) = \frac{1}{p}\log\frac{x-x_{0}}{\alpha}, \qquad w(x) = s'(x)^{2} = \frac{1}{p^{2}(x-x_{0})^{2}}, \]
so $w$ is of the form \eqref{eq:inverse-square} with $c = p^{-2}$. 

Conversely, suppose that $w$ is of the form \eqref{eq:inverse-square} and put $y = \pm(x-x_{0})>0$. 
This affine change of the $\nabla$-affine coordinate has unit Jacobian, so $y$ is again $\nabla$-affine and the metric becomes $g = c\,y^{-2}\,dy^{2}$.
We apply Lemma~\ref{lem:normal-form} in the coordinate $y$, with all derivatives taken with respect to $y$. 
Normalizing $\theta(y) = y$, so that $\theta' \equiv 1$, $A$ is the arithmetic mean in $y$. 
From $\eta'(y) = c\,y^{-2}/\theta'(y) = c\,y^{-2}$, we obtain $\eta(y) = -c\,y^{-1}$ up to an additive constant.
Thus $M$ is the harmonic mean.
Also, $s'(y) = \sqrt{c}\,y^{-1}$ gives $s(y) = \sqrt{c}\,\log y$, so $M_g$ is the geometric mean. 
The invariance identity holds in the classical form: since $A(y_1,y_2)\,M(y_1,y_2) = y_1y_2$,
\[ M_g\left(A(y_1,y_2),M(y_1,y_2)\right) = \sqrt{A(y_1,y_2)\,M(y_1,y_2)} = \sqrt{y_1y_2} = M_g(y_1,y_2). \qedhere \]
\end{proof}

The one-dimensional classification produces necessary conditions in higher dimensions by restriction to curves which are autoparallel with respect to both members of the dual pair.
Such submanifolds are called \emph{doubly autoparallel} by Ohara, Ishi, and Tsuchiya~\cite{ohara2024doubly}; by \eqref{eq:dual-average}, autoparallelism with respect to $\nabla$ and $\nabla^{g}$, as assumed below, is equivalent to autoparallelism with respect to $\nabla$ and $\nabla^{*}$.

\begin{Thm}[Restriction to autoparallel curves]\label{thm:restriction}
Let $(\mathcal{P},g,\nabla,\nabla^{*})$ be a torsion-free dualistic structure; that is, assume that both $\nabla$ and $\nabla^*$ are torsion-free.  
Let $\mathcal{N} \subset \mathcal{P}$ be a one-dimensional embedded submanifold that is autoparallel with respect to both $\nabla$ and $\nabla^{g}$, that is, $\nabla_{X}Y$ and $\nabla^{g}_{X}Y$ are tangent to $\mathcal{N}$ whenever $X,Y$ are vector fields tangent to $\mathcal{N}$. 
Then, \\
(i) $\mathcal{N}$ is autoparallel with respect to $\nabla^{*} = 2\nabla^{g}-\nabla$.\\
(ii) $\nabla^{*}|_{\mathcal{N}}$ is the $g|_{\mathcal{N}}$-dual of $\nabla|_{\mathcal{N}}$, and $\nabla^{g}|_{\mathcal{N}}$ is the Levi--Civita connection of $g|_{\mathcal{N}}$.\\
(iii) Let $\mathcal{U} \subset \mathcal{N}$ be a connected open arc which is geodesically convex with respect to each of the induced connections
$\nabla|_{\mathcal{U}},\quad \nabla^*|_{\mathcal{U}},\quad \nabla^g|_{\mathcal{U}}$; 
that is, every pair of points in $\mathcal{U}$ is joined by a unique geodesic segment of each induced connection that remains in $\mathcal{U}$. 
Assume moreover that, for every $P,Q\in \mathcal{U}$, the corresponding ambient geodesic segment is unique for each of $\nabla$, $\nabla^*$, and $\nabla^g$.
Then, for all $P,Q\in \mathcal{U}$, the ambient midpoint means
$A(P,Q)$, $M(P,Q)$, and $M_g(P,Q)$ lie in $\mathcal{U}$ and
coincide with the corresponding midpoint means
$A_{\mathcal{U}}(P,Q)$, $M_{\mathcal{U}}(P,Q)$, and $M_{g,\mathcal{U}}(P,Q)$ of the induced
one-dimensional structure on $\mathcal{U}$.

Consequently, if the invariance identity holds on $\mathcal P$, then
\[ M_{g,\mathcal{U}}\left(A_{\mathcal{U}}(P,Q),M_{\mathcal{U}}(P,Q)\right) = M_{g,\mathcal{U}}(P,Q), \quad P,Q \in \mathcal{U}.\]
After identifying $\mathcal{U}$ with an open interval by a $\nabla|_{\mathcal{U}}$-affine coordinate, Theorem~\ref{thm:rigidity} applies to the induced structure on $\mathcal{U}$.
\end{Thm}

\begin{proof}
(i) is immediate from $\nabla^{*} = 2\nabla^{g}-\nabla$, that is, from \eqref{eq:dual-average}. 

For (ii), evaluate the defining identity \eqref{eq:duality} on vector fields tangent to $\mathcal{N}$.
The identity involves only the restricted data, and the $g|_{\mathcal{N}}$-dual connection is unique.
Also, $\nabla^{g}|_{\mathcal{N}}$ is torsion-free and $g|_{\mathcal{N}}$-metric, so it is the Levi--Civita connection of $g|_{\mathcal{N}}$. 

For (iii), fix $D \in \{\nabla,\nabla^*,\nabla^g\}$ 
and let $\gamma_D^{\mathcal{U}} : [0,1]\to \mathcal{U}$ be the geodesic of the induced
connection joining $P$ to $Q$. 
Since $\mathcal{U}\subset\mathcal{N}$ is $D$-autoparallel, $\gamma_D^{\mathcal{U}}$ is also an ambient $D$-geodesic. 
By the uniqueness of the ambient geodesic,
$\gamma_D^{\mathcal{U}}$ coincides with the ambient geodesic joining $P$ to $Q$. 
Hence their midpoints agree and belong to $\mathcal{U}$.

In particular,
\[ A(P,Q) = A_{\mathcal{U}} (P,Q),\quad M(P,Q) = M_{\mathcal{U}} (P,Q),\quad M_g(P,Q) = M_{g,\mathcal{U}}(P,Q). \]
Since $A_{\mathcal{U}} (P,Q)$ and $M_{\mathcal{U}} (P,Q)$ also belong to $\mathcal{U}$, the same argument applies to the pair $\left(A_{\mathcal{U}}(P,Q),M_{\mathcal{U}}(P,Q)\right)$. 
Therefore the ambient invariance identity restricts to the stated intrinsic invariance identity on $\mathcal{U}$.
\end{proof}

\begin{Rem}\label{rem:examples-revisited}
(1) Autoparallel curves as in Theorem~\ref{thm:restriction} arise as one-dimensional components of common fixed-point sets of transformations that are simultaneously $g$-isometric and $\nabla$-affine, by the standard geodesic-uniqueness argument; such transformations preserve $\nabla^{g}$ and, since $\nabla^{*}$ is determined by $(g,\nabla)$, also $\nabla^{*}$.\\
(2) On the scalar ray $\{tI_m: t>0\}$ of Section~\ref{sec:nakamura}, the coordinate $t$ is $\nabla^{E}$-affine and $w(t) = \operatorname{tr}\left((tI)^{-1}I(tI)^{-1}I\right) = m\,t^{-2}$, which is exactly the form \eqref{eq:inverse-square} with $c = m$ and $x_{0} = 0$. 
The restricted means are the arithmetic, harmonic, and geometric means, consistent with the invariance identity valid on all of $\SPD(m)$.\\
(3) The scalar ray $\{tI_{2}: t>0\}$ of Section~\ref{sec:counterexample} is the fixed-point set of the linear map $(x,y,z)\mapsto(y,x,-z)$, which preserves $\Phi$ and is therefore a $\nabla^{E}$-affine $g$-isometry. 
On this ray, $t$ is $\nabla^{E}$-affine and $w(t) = 2t^{2}$, which is neither constant nor of the form $c(t-t_{0})^{-2}$; equivalently, $s(t) = t^{2}/\sqrt{2}$ gives $\varphi(u) = \theta\circ s^{-1}(u)\propto\sqrt{u}$, which is neither affine nor exponential. 
By Theorem~\ref{thm:rigidity} the invariance identity fails on the ray, hence for the Hessian structure of Section~\ref{sec:counterexample}; and since the iteration converges automatically in dimension one, $A\otimes M \ne  M_g$ on the ray. 
This gives another proof of the conclusion of Section~\ref{sec:counterexample} without numerical computation.
The numerical values give the size of the difference.\\
(4) For a product of one-dimensional dualistic structures, the midpoint means $A$, $M$, and $M_g$ act componentwise.
Hence the invariance identity holds if and only if each factor is of one of the two types in Theorem~\ref{thm:rigidity}(c).
In the Hessian setting, this includes $\nabla=\nabla^E$ with $\Phi(x_1,\ldots,x_d)=\sum_{i=1}^{d}\Phi_i(x_i)$.\\
(5) For the Euclidean structures of Section~\ref{sec:parallel-cubic} with $\nabla$ flat, the commuting $g$-symmetric operators $K_X$ are simultaneously diagonalizable in an orthonormal basis $(f_i)$.
The symmetry of $K$ then gives
\[ K(f_i,f_j)=0 \quad (i\ne j),  \quad  K(f_i,f_i) = c_i f_i.\]
Hence, the structure is the orthogonal product of the one-dimensional structures $(\mathbb R,du^2,\Gamma\equiv c_i)$,
whose normal form in Lemma~\ref{lem:normal-form} has, up to affine transformations, the generators $(\varphi,\psi)=(\chi_{c_i},\chi_{-c_i})$.
These are precisely the exponential pairs of Theorem~\ref{thm:DP}; equivalently, up to gauge, 
they are factors of the two types in Theorem~\ref{thm:rigidity}(c), consistently with item~(4).
\end{Rem}

\section{A rigidity conjecture for Euclidean dual pairs}\label{sec:conjecture}

Section~\ref{sec:parallel-cubic} shows that every constant totally symmetric cubic form on $\mathbb{R}^d$ yields a Euclidean dual pair satisfying the invariance identity \eqref{eq:sum-invariance} for all sufficiently close pairs.
We conjecture that these are all the Euclidean dual pairs with this property.
We give a pointwise necessary condition of fourth order.
Using this condition, we prove the conjecture in dimension one and for cubic forms that are scalar multiples of a constant one.

Let $\Omega \subset \mathbb{R}^d$ be a connected open set with the Euclidean metric $g = \langle\cdot,\cdot\rangle$, so that $\nabla^E = \nabla^g$ and $M_g(P,Q) = \frac{P+Q}{2}$.
For a smooth totally symmetric $3$-tensor field $C$ on $\Omega$, define $K$, $\nabla = \nabla^E + K$, and $\nabla^* = \nabla^E - K$ as in Section~\ref{sec:parallel-cubic}.
Here we do not assume that the coefficients are constant.
We call $C$ \emph{constant} if $D_vC = 0$ for every $v \in \mathbb{R}^d$.
As in that section, $\nabla$ and $\nabla^*$ are torsion-free and $g$-dual.
By Section~\ref{sec:setting}, every torsion-free connection on $\Omega$ with torsion-free $g$-dual is obtained in this way from $C = -\frac{1}{2}\nabla g$.
Let $D = \nabla^E + \Gamma$ be a torsion-free connection on $\Omega$.
The map $(P,v) \mapsto \left(P,\operatorname{Exp}^{D}_{P}(v)\right)$ is a diffeomorphism from an open neighborhood of $\Omega \times \{0\}$ onto an open neighborhood $W_D$ of the diagonal $\Delta_\Omega$ of $\Omega \times \Omega$ \cite[Chapter~III, Proposition~8.1]{kobayashi-nomizu-1963}.
By \eqref{eq:midpoint-exp}, we define the \emph{local midpoint map} $M_D$ on $W_D$.
It is symmetric on a neighborhood of $\Delta_\Omega$.
Write $A \coloneqq M_{\nabla}$ and $M \coloneqq M_{\nabla^*}$.
We say that $C$ has the \emph{local midpoint invariance property} if
\begin{equation}\label{eq:local-sum}
A(P,Q) + M(P,Q) = P + Q
\end{equation}
holds on an open neighborhood of $\Delta_\Omega$ in $W_\nabla \cap W_{\nabla^*}$.
This is the invariance identity \eqref{eq:Mg-inv} near the diagonal.
By Theorem~\ref{thm:local-convergence}, it implies that $A \otimes M = M_g$ for all sufficiently close pairs.
Every constant $C$ on $\mathbb{R}^d$ has the property by \eqref{eq:sum-invariance}.

\begin{Conj}\label{conj:rigidity}
Let $\Omega \subset \mathbb{R}^d$ be a connected open set and let $C$ be a smooth totally symmetric $3$-tensor field on $\Omega$ with the local midpoint invariance property.
Then $C$ is constant.
\end{Conj}

For a smooth symmetric $(1,2)$-tensor field $\Gamma$ on $\Omega$, we write $D_v\Gamma$ and $D^2\Gamma$ for its first and second directional derivatives, and we define, for $x \in \Omega$ and $h \in \mathbb{R}^d$, with all tensors evaluated at $x$,
\begin{equation}\label{eq:E4}
\mathcal{E}_{\Gamma}(x;h) \coloneqq 2\,\Gamma\left(h, (D_h\Gamma)(h,h)\right) + 5\,(D_{\Gamma(h,h)}\Gamma)(h,h) - 4\,(D_h\Gamma)\left(h, \Gamma(h,h)\right) .
\end{equation}

\begin{Thm}[Fourth-order obstruction]\label{thm:E4}
Let $\Gamma_1, \Gamma_2$ be smooth symmetric $(1,2)$-tensor fields on $\Omega$, let $A_1, A_2$ be the local midpoint maps of $\nabla^{E} + \Gamma_1$ and $\nabla^{E} + \Gamma_2$, and suppose that $A_1(P,Q) + A_2(P,Q) = P + Q$ on an open neighborhood of $\Delta_{\Omega}$.
Then \\
(1) $\Gamma_2 = -\Gamma_1$ on $\Omega$; \\
(2) $\mathcal{E}_{\Gamma_1}(x;h) = 0$ for all $x \in \Omega$ and $h \in \mathbb{R}^d$.
\end{Thm}

\begin{Thm}\label{thm:rankone}
Let $\Omega \subset \mathbb{R}^d$ be a connected open set, let $C_1 \ne 0$ be a constant totally symmetric $3$-tensor on $\mathbb{R}^d$, let $f \in C^{\infty}(\Omega)$, and let $K$ be the $(1,2)$-tensor field associated with $C = fC_1$.
If $\mathcal{E}_{K} \equiv 0$ on $\Omega$, then $f$ is constant.
Consequently, Conjecture~\ref{conj:rigidity} holds for $d = 1$ and, in every dimension, for $C = fC_1$ with $C_1$ constant.
\end{Thm}

\begin{proof}[Proof of Theorem~\ref{thm:E4}]
Let $\Gamma$ be a smooth symmetric $(1,2)$-tensor field, $A_\Gamma$ the local midpoint map of $\nabla^E + \Gamma$, and $o \in \Omega$.
Put $P = o - h/2$ and $Q = o + h/2$.
Write the local geodesic from $P$ to $Q$ as $\gamma(t) = o + \sigma h + u(t)$ with $\sigma \coloneqq t - \frac{1}{2}$ and $u(0) = u(1) = 0$.
Then $\ddot{u} = -\Gamma_\gamma(h + \dot{u}, h + \dot{u})$.
We expand $\Gamma_\gamma$ at $o$ and write $u = u_2 + u_3 + u_4 + O(\|h\|^5)$ by homogeneity in $h$.
We solve the resulting two-point boundary value problems at each degree, as in Section~\ref{sec:parallel-cubic}.
With $k \coloneqq \Gamma(h,h)$ and $n \coloneqq 2\Gamma(h,k) - (D_h\Gamma)(h,h)$, we obtain
\[ u_2 = \left(\frac{1}{8} - \frac{\sigma^2}{2}\right)k, \qquad u_3 = \left(\frac{\sigma^3}{6} - \frac{\sigma}{24}\right)n, \qquad \ddot{u}_4 = B_0 + B_2\sigma^2, \]
where
\begin{align*}
B_0 &= \frac{1}{12}\Gamma(h,n) - \frac{1}{8}(D_k\Gamma)(h,h), \\
B_2 &= -\Big(\Gamma(k,k) + \Gamma(h,n) - 2(D_h\Gamma)(h,k) - \frac{1}{2}(D_k\Gamma)(h,h) \\
&\qquad\quad + \frac{1}{2}(D^2\Gamma)(h,h)(h,h)\Big),
\end{align*}
and hence $u_4(1/2) = -\frac{1}{8}B_0 - \frac{1}{192}B_2$.
Since $A_\Gamma(o - h/2, o + h/2) = \gamma(1/2)$ is an even function of $h$ by the symmetry of $A_\Gamma$, this gives
\begin{equation}\label{eq:expansion4}
A_{\Gamma}\left( o - \frac{h}{2},\, o + \frac{h}{2} \right) = o + \frac{1}{8}\,\Gamma(h,h) + \frac{1}{384}\, T_{\Gamma}(h) + O\left(\|h\|^6\right),
\end{equation}
\begin{align*}
T_{\Gamma}(h) \coloneqq\ & -4\,\Gamma\left(h, \Gamma(h, \Gamma(h,h))\right) + 2\,\Gamma\left(\Gamma(h,h), \Gamma(h,h)\right) \\
& + (D^2\Gamma)(h,h)(h,h) + \mathcal{E}_\Gamma(o;h) .
\end{align*}
Adding \eqref{eq:expansion4} for $\Gamma_1$ and $\Gamma_2$ and using the hypothesis, we see that the homogeneous parts of degree $2$ and $4$ vanish.
The degree-$2$ part gives $\Gamma_1(h,h) + \Gamma_2(h,h) = 0$ for all $h$, hence (1) by symmetry.
With $\Gamma_2 = -\Gamma_1$, the first three terms of $T_{\Gamma_1} + T_{-\Gamma_1}$ cancel because they are of odd degree in $\Gamma_1$.
Also, $\mathcal{E}_{-\Gamma_1} = \mathcal{E}_{\Gamma_1}$.
Thus the degree-$4$ part is $\frac{1}{192}\mathcal{E}_{\Gamma_1}(o;h)$, which proves (2).
\end{proof}

For $d = 1$ and $\Gamma = a(x)$, \eqref{eq:expansion4} reads
\[ A_\Gamma\left(o - \frac{h}{2}, o + \frac{h}{2}\right) = o + \frac{a}{8}h^2 + \frac{a'' + 3aa' - 2a^3}{384}h^4 + O(h^6), \]
in accordance with the expansion of the quasi-arithmetic mean with generator $\theta$ satisfying $\theta''/\theta' = a$, and $\mathcal{E}_\Gamma(x;h) = 3a(x)a'(x)h^4$.

\begin{proof}[Proof of Theorem~\ref{thm:rankone}]
Let $K_1$ be the $(1,2)$-tensor associated with $C_1$.
Since $K = fK_1$ and $D_vK = df(v)K_1$, \eqref{eq:E4} becomes
\begin{equation}\label{eq:E4-rankone}
\mathcal{E}_{K}(x;h) = f(x)\Big( 5\, df_x\left(K_1(h,h)\right)\, K_1(h,h) - 2\, df_x(h)\, K_1\left(h, K_1(h,h)\right) \Big) .
\end{equation}
We claim that if a nonzero linear form $\ell = \langle n,\cdot\rangle$ satisfies
\begin{equation}\label{eq:alg}
5\,\ell\left(K_1(h,h)\right)K_1(h,h) = 2\,\ell(h)\,K_1\left(h,K_1(h,h)\right) \qquad \text{for all } h \in \mathbb{R}^d,
\end{equation}
then $C_1 = 0$.
Normalize $\|n\| = 1$ and put $H \coloneqq n^\perp$.
For $h \in H$, \eqref{eq:alg} gives $\ell(K_1(h,h))K_1(h,h) = 0$, hence $C_1(h,h,n) = 0$, and by polarization $C_1(H,H,n) = 0$.
Thus $K_H \coloneqq K_1|_{H \times H}$ takes values in $H$, and, for $h \in H$,
\[ K_1(h,n) = \langle b,h\rangle n, \qquad K_1(n,n) = b + cn, \]
where $b \in H$ is defined by $\langle b,h\rangle = C_1(h,n,n)$ for $h \in H$ and $c \coloneqq C_1(n,n,n)$.
Insert $h = h_0 + sn$ with $h_0 \in H$ and $s \in \mathbb{R}$.
Put $q \coloneqq K_H(h_0,h_0)$ and $\beta \coloneqq \langle b,h_0\rangle$.
Then $K_1(h,h) = (q + s^2b) + s(2\beta + sc)n$.
Comparing the coefficients of $s^1$ and $s^3$ in the $H$-component of \eqref{eq:alg}, and of $s^2$ in its $n$-component, yields, for all $h_0 \in H$,
\[ K_H(h_0,q) = 5\beta q, \qquad K_H(h_0,b) = 3\beta b, \qquad \langle b,q\rangle = 8\beta^2 . \]
Pairing the second relation with $h_0$ and using the total symmetry of $C_1$, we obtain $\langle b,q\rangle = 3\beta^2$.
Hence $\beta = 0$ for all $h_0$, that is, $b = 0$.
Pairing the first relation with $h_0$ then gives $\|q\|^2 = 0$, that is, $K_H = 0$.
Finally, the $n$-component reduces to $5s^4c^2 = 2s^4c^2$, so $c = 0$.
Hence $C_1 = 0$, proving the claim.

By the claim and \eqref{eq:E4-rankone}, $df_x = 0$ at every $x$ with $f(x) \ne 0$.
Thus $f$ is locally constant on the open set $\{f \ne 0\}$.
Each nonempty level set $\{f = a\}$ with $a \ne 0$ is open and closed in the connected set $\Omega$.
Hence $f$ is constant.
The last assertion follows from Theorem~\ref{thm:E4}, since for $d = 1$ every $C$ is of the form $fC_1$ with $C_1 = dx^3$.
\end{proof}

For $d = 1$, this conclusion agrees with Theorem~\ref{thm:rigidity}.
Indeed, with $C = a(x)\,dx^3$ and $g = dx^2$, a $\nabla$-affine coordinate $\theta$ satisfies $\theta' = e^{\int^x a}$, and the metric coefficient is $\widetilde{w} = 1/(\theta')^2$.
The two alternatives in Theorem~\ref{thm:rigidity}(c) are exactly $a \equiv 0$ and $a \equiv \pm c^{-1/2}$.

\begin{Rem}\label{rem:conjecture}
(1) The duality between $\nabla$ and $\nabla^*$ cannot be dropped from Conjecture~\ref{conj:rigidity}.
Let $d = d_1 + d_2$, write $x = (x',x'') \in \mathbb{R}^{d_1} \times \mathbb{R}^{d_2}$, let $F \colon \mathbb{R}^{d_1} \to \mathbb{R}^{d_2}$ be smooth, and let $\nabla_1$, $\nabla_2$ be the flat connections with affine coordinates $\theta(x) = (x', x'' + F(x'))$ and $\eta(x) = (x', x'' - F(x'))$.
Their midpoint maps are the quasi-arithmetic means $M_\theta$ and $M_\eta$, and a direct computation gives $M_\theta(P,Q) + M_\eta(P,Q) = P + Q$ on all of $\mathbb{R}^d \times \mathbb{R}^d$.
The Christoffel tensors $\pm\left(0, D^2F(X',Y')\right)$ are not constant unless $F$ is a polynomial of degree at most two.
Also, $\nabla_2$ is not the Euclidean dual of $\nabla_1$ unless $D^2F \equiv 0$.
Indeed, duality would require $\langle \Gamma_1(X,Y),Z\rangle$ to be symmetric in $(Y,Z)$.
Thus, without duality, \eqref{eq:local-sum} admits infinite-dimensional families of nonconstant solutions in every dimension $d \ge 2$, even among flat connections. \\
(2) It remains open whether $\mathcal{E}_K \equiv 0$ alone implies that $C$ is constant.
An affirmative answer would prove Conjecture~\ref{conj:rigidity}.
The local midpoint invariance property gives further pointwise conditions of every even order $2m \ge 6$.
These follow from the terms of degree $2m$ in the expansion of $A + M$ at the diagonal.
It is not known whether these further conditions are needed.
\end{Rem}

\appendix

\section{Remarks on dualistic structures}\label{app:remarks}

In this appendix, we give further explanations of the dualistic structure in Section~\ref{sec:setting} and compare it with related results in the literature.
These remarks are not needed for the main arguments.

\begin{Rem}[Duality versus metric compatibility]\label{rem:duality-vs-metricity}
Duality with respect to $g$ and metric compatibility of either member of the dual pair are different notions.
Indeed, the defining identity \eqref{eq:duality} gives
\begin{equation}\label{eq:duality-nonmetricity}
  (\nabla_X g)(Y,Z)   = g\left(Y,(\nabla_X^{*}-\nabla_X)Z\right),  \quad  (\nabla_X^{*}g)(Y,Z) = -(\nabla_Xg)(Y,Z).
\end{equation}
These identities are standard in the theory of conjugate connections;
see \cite{AmariNagaoka2000,nomizu-1992}.
Since $g$ is nondegenerate, \eqref{eq:duality-nonmetricity} implies the equivalence
\begin{equation}\label{eq:metric-self-dual}
  \nabla g = 0 \quad \Longleftrightarrow \quad \nabla  =  \nabla^{*}  \quad \Longleftrightarrow \quad  \nabla^{*} g = 0.
\end{equation}

Thus, the statement that the dual of a metric connection is again metric is correct only with respect to the same metric $g$ used in the duality relation \eqref{eq:duality} (see balanced metrics~\cite{thanwerdas2019exploration}).
In that case, the two connections in fact coincide.
Metricity with respect to some other metric does not imply \eqref{eq:metric-self-dual}.
\end{Rem}

\begin{Rem}[The arithmetic mean of a dual pair]\label{rem:mean-connection}
Even when neither connection is $g$-metric, their arithmetic mean $\overline{\nabla} \coloneqq  (\nabla + \nabla^{*})/2$ is always $g$-metric.
If both $\nabla$ and $\nabla^{*}$ are torsion-free, then $\overline{\nabla}$ is also torsion-free.
Therefore, the fundamental theorem of Riemannian geometry yields $\overline{\nabla} = \nabla^{g}$ and hence $\nabla^{*} = 2\nabla^{g} - \nabla$, which is \eqref{eq:dual-average}.  
In particular, if either member of a torsion-free dual pair is $g$-metric, then
\[  \nabla  =  \nabla^{*}  =  \nabla^{g}, \quad A  =  M  =  M_g. \]

If torsion is allowed but either member of the dual pair is still $g$-metric, 
\eqref{eq:metric-self-dual} gives $\nabla  =  \nabla^{*}$ and hence $A  =  M$, but this common metric connection need not be the Levi--Civita connection.
Consequently, $A  =  M_g$ does not follow in general.
\end{Rem}

\begin{Rem}[Torsion, Codazzi coupling, and statistical manifolds]\label{rem:torsion-codazzi}
The duality relation \eqref{eq:duality} also relates the antisymmetric part of $\nabla g$ to the difference of the torsion tensors.
If
\[ T^D(X,Y) \coloneqq D_{X} Y - D_{Y} X - [X,Y]\]
denotes the torsion of an affine connection $D$, then
\begin{equation}\label{eq:dual-torsion-codazzi}
  g\left(T^{\nabla^*}(X,Y)-T^\nabla(X,Y),Z\right)  = (\nabla_{X}g)(Y,Z)-(\nabla_{Y}g)(X,Z).
\end{equation}
This is the torsion identity for conjugate connections
\cite[Lemma~2, Eq.~(3)]{zhang2019hessian}.  
The metric $g$ and the connection $\nabla$ are said to be \emph{Codazzi-coupled} when
\begin{equation}\label{eq:codazzi-coupling}
  (\nabla_{X} g)(Y,Z)=(\nabla_{Y} g)(X,Z)
\end{equation}
for all vector fields $X,Y,Z$.  
Since $(\nabla_Xg)(Y,Z)$ is symmetric in $Y,Z$, condition~\eqref{eq:codazzi-coupling} is equivalent to the total symmetry of the cubic tensor $\nabla g$.  
By the nondegeneracy of $g$, Equation~\eqref{eq:dual-torsion-codazzi} implies that  $(g,\nabla)$ is Codazzi-coupled if and only if $T^{\nabla^*} = T^\nabla$.
In particular, if either member of the conjugate pair is torsion-free,
Codazzi coupling is equivalent to torsion-freeness of the other member.
Thus, the usual statistical-manifold assumption that both conjugate connections are torsion-free is equivalently expressed as torsion-freeness of one member together with Codazzi coupling.

Matumoto's realization theorem further shows that every such statistical structure is induced by a globally defined smooth contrast function on $\mathcal{P} \times \mathcal{P}$ through its second- and third-order derivatives along the diagonal \cite[Theorem~1]{matumoto1993any}.
This realization is generally nonunique and implies neither flatness nor the midpoint invariance \eqref{eq:Mg-inv} studied in this paper.
\end{Rem}

\begin{Rem}[Curvature-freeness and affine coordinates]\label{rem:curvature-freeness}
We explain the roles of curvature-freeness and torsion-freeness separately.
We call an affine connection $D$ {\it curvature-free} when $R^D = 0$, and {\it flat} when it is both curvature-free and torsion-free.  
The definition $A = M_\nabla$ does not require $\nabla$ to be curvature-free.
The simultaneous vanishing of curvature and torsion gives affine coordinates which linearize all $\nabla$-geodesics.
Indeed, if $R^\nabla=0$ and $T^\nabla=0$, then locally there are $\nabla$-affine coordinates $\theta$ in which $\theta(A(P,Q)) = (\theta(P)+\theta(Q)) /2$. 
The ordinary arithmetic formula $A(P,Q)= (P+Q)/2$ is the particular case in which the given matrix coordinates are such affine coordinates.  
If $\nabla$ is torsion-free and $R^{\nabla}$ is nonzero at a point, then $A$ is still a geodesic midpoint mean, but it admits no such quasi-arithmetic representation on a neighborhood of that point.
The torsion-free hypothesis is essential here: the connection
\[ \nabla^{\mathrm{sym}}_X Y\coloneqq\nabla_XY-\frac12T^\nabla(X,Y) \]
is torsion-free and has the same affinely parametrized geodesics, hence the same midpoint maps, as $\nabla$.
Thus a smooth local midpoint map determines $\nabla^{\mathrm{sym}}$, rather than the full connection $\nabla$; see Remark~\ref{rem:qam-comparison}.

In \cite{zhang2007note}, the curvatures of a pair of $g$-dual connections satisfy 
\[ g\left(R^{\nabla}(X,Y)Z,W\right)  =  -g\left(Z,R^{\nabla^{*}}(X,Y)W\right).\]   
Consequently, $R^\nabla = 0$ if and only if   $R^{\nabla^*}=0$, 
although $R^{\nabla}$ and $R^{\nabla^*}$ need not be equal as curvature tensors.  
Curvature-freeness alone yields local parallel frames, whereas affine coordinates with vanishing connection coefficients require torsion-freeness as well.  
Thus, when both connections are curvature-free and torsion-free, each generally admits its own local affine coordinate system.
Duality does not make the two systems coincide.  
In particular, if $R^{\nabla} \ne 0$, then $R^{\nabla^*} \ne 0$. 
Neither curvature-freeness nor coordinate linearization is used in the invariance criterion of Theorem~\ref{thm:symmetry}.  
It requires only the stated local uniqueness and symmetry hypotheses for the relevant midpoint constructions.
Convergence is used in Theorem~\ref{thm:invariance} only to identify the invariant midpoint with the Gauss composition. 
\end{Rem}

\vspace{1pc}

\noindent{\it Acknowledgements}  The author K.O. has received funding from JSPS KAKENHI Grant Number JP22K13928.  

\vspace{1pc}

\noindent{\bf Use of generative AI} \ During the preparation of this work, the authors used Claude and ChatGPT to assist with exploration, structuring, drafting, and refining the text. 
After using these tools, the authors reviewed and edited the content as needed. 
The authors take full responsibility for the content and integrity of the publication.

\bibliographystyle{plain}
\bibliography{RieCompoundMeanBIB}

\end{document}